\documentclass[12pt]{article}

\usepackage[margin=1in]{geometry}
\usepackage{amscd}
\usepackage{amssymb}
\usepackage{amsmath}
\usepackage{amsthm}
\usepackage{enumerate}
\usepackage{tikz}
\usepackage{wrapfig, framed, caption}
\usepackage{float}
\usetikzlibrary{arrows}
\usepackage[font=small,labelfont=bf]{caption}
\usepackage{xcolor}
\usepackage{mathtools}
\usepackage[colorlinks=true, linkcolor=blue, citecolor=blue,
pagebackref=true]{hyperref}
\usepackage{fouriernc}
\usepackage{ulem}

\newtheorem{theorem}{Theorem}

\newtheorem*{theoremY*}{Theorem Y}

\newtheorem*{theoremAB*}{Theorem AB}

\newtheorem*{linearformsmtp*}{Mass transference principle for linear forms}

\newtheorem*{corollary*}{Corollary}
\newtheorem{proposition}{Proposition}
\newtheorem{lemma}{Lemma}

\newtheorem*{claim*}{Claim}
\newtheorem{conjecture}{Conjecture}

\theoremstyle{definition}
\newtheorem{definition}{Definition}
\theoremstyle{remark}

\newtheorem*{remark*}{Remark}

\renewcommand{\Bbb}[1]{\mathbb{#1}}
\newcommand{\N}{{\Bbb N}}         
\newcommand{\I}{{\Bbb I}}
\newcommand{\R}{{\Bbb R}}        
\newcommand{\Z}{{\Bbb Z}}         

\newcommand{\cA}{{\cal A}}

\newcommand{\cH}{{\cal H}}

\newcommand{\x}{\mathbf{x}}
\newcommand{\y}{\mathbf{y}}
\newcommand{\p}{\mathbf{p}}
\newcommand{\q}{\mathbf{q}}
\newcommand{\br}{\mathbf{r}}

\newcommand{\0}{\mathbf{0}}

\DeclareMathOperator{\dimh}{\dim_H}

\newcommand{\NN}{\mathbb{N}}

\newcommand{\RR}{\mathbb{R}}

\newcommand{\ZZ}{\mathbb{Z}}

\newcommand{\calA}{\mathcal A}

\newcommand{\calH}{\mathcal H}

\newcommand{\bx}{\mathbf x}

\newcommand{\by}{\mathbf y}

\newcommand{\bz}{\mathbf z}

\DeclarePairedDelimiter{\abs}{\lvert}{\rvert}
\DeclarePairedDelimiter{\set}{\lbrace}{\rbrace}

\DeclarePairedDelimiter{\parens}{\lparen}{\rparen}
\DeclarePairedDelimiter{\brackets}{\lbrack}{\rbrack}

\DeclareMathOperator{\sgn}{sgn}

\def\eps{{\varepsilon}}

\title{Independence inheritance and Diophantine
  approximation for systems of linear forms}

\author{Demi Allen\footnote{research for this project supported by the Heilbronn Institute for
    Mathematical Research} \\ University of Warwick \and Felipe
  A.~Ram\'{\i}rez \\ Wesleyan University}
\date{\footnotesize{\it Dedicated to Victor Beresnevich and Sanju
    Velani}}
\begin{document}
\frenchspacing
\maketitle
\begin{abstract}

  The classical Khintchine--Groshev theorem is a generalization of
  Khintchine's theorem on simultaneous Diophantine approximation, from
  approximation of points in $\RR^m$ to approximation of systems of
  linear forms in $\RR^{nm}$. In this paper, we present an
  inhomogeneous version of the Khintchine--Groshev theorem which does
  not carry a monotonicity assumption when $nm>2$.  Our results bring
  the inhomogeneous theory almost in line with the homogeneous theory,
  where it is known by a result of Beresnevich and Velani (2010) that
  monotonicity is not required when $nm>1$. That result resolved a
  conjecture of Beresnevich, Bernik, Dodson, and Velani (2009), and
  our work resolves almost every case of the natural inhomogeneous
  generalization of that conjecture. Regarding the two cases where
  $nm=2$, we are able to remove monotonicity by assuming extra
  divergence of a measure sum, akin to a linear forms version of the
  Duffin--Schaeffer conjecture. When $nm=1$ it is known by work of
  Duffin and Schaeffer (1941) that the monotonicity assumption cannot
  be dropped.
  
  The key new result is an \emph{independence inheritance phenomenon};
  the underlying idea is that the sets involved in the
  $((n+k)\times m)$-dimensional Khintchine--Groshev theorem
  ($k\geq 0$) are always $k$-levels more probabilistically independent
  than the sets involved the $(n\times m)$-dimensional theorem. Hence,
  it is shown that Khintchine's theorem itself underpins the
  Khintchine--Groshev theory.
\end{abstract}

\newpage
{\footnotesize
\tableofcontents
}

\section{Introduction}

\subsection{The Lebesgue measure theory}

For a sequence of balls $\Psi :=(B_q)_{q \in \N}\subset\RR^m/\ZZ^m$, and
$n\geq 1$, let $\calA_{n,m}(\Psi)$ denote the set of $\x \in \I^{nm}$
such that
\begin{align*}
\q\x+\p \in B_{\abs{\q}}
\end{align*}
holds for infinitely many pairs
$(\p,\q) \in \Z^{m} \times \Z^{n}$,
where $\abs{\q}:=\max_{1 \leq i \leq n}|q_i|$ is the maximum norm. If
all the balls have a common center $\y$, and their radius is given by
a function $\psi(\abs{\q})$, then we write $\calA_{n,m}^\y (\psi)$
instead of $\calA_{n,m}(\Psi)$. Throughout we write
$\I := [0,1] = \RR/\ZZ$ and, where not explicitly specified, $n$ and
$m$ will always be integers such that $n,m \geq 1$. Whenever we refer
to balls, we will mean balls with respect to the maximum norm.

Many seminal results in Diophantine approximation have had to do with
the sets $\calA_{n,m}^\y(\psi)$, in particular their metric
properties, typically meaning their Lebesgue measure
$\abs*{\calA_{n,m}^\y(\psi)}$ and their Hausdorff measures
$\calH^f\parens*{\calA_{n,m}^\y(\psi)}$ . The classical Lebesgue
theory for these sets is summarized by the following statement.
\theoremstyle{plain} \newtheorem*{ikg}{(Inhomogeneous)
  Khintchine--Groshev Theorem}
\begin{ikg}
  Let $n,m\geq 1$. Then for any $\psi:\NN\to \RR_{\geq 0}$ and
  $\y\in \RR^m$, we have
  \begin{equation}\label{eq:ikg}
    \abs*{\calA_{n,m}^\y(\psi)} =
    \begin{cases}
      0 &\textrm{if } \sum_{q=1}^\infty q^{n-1}\psi(q)^m < \infty \;, \\[2ex]
      1 &\textrm{if } \sum_{q=1}^\infty q^{n-1}\psi(q)^m =\infty \quad\dots\quad \textrm{ \textbf{and $\psi$ is monotonic.}}
    \end{cases}
  \end{equation}
\end{ikg}
\begin{remark*}[On terminology]
  The term ``homogeneous approximation'' refers to the case $\y=\0$,
  and ``inhomogeneous approximation'' refers to the general
  case. ``Simultaneous approximation'' refers to cases where $n=1$ and
  $m>1$. ``Dual approximation'' refers to cases where $n>1$ and
  $m=1$. ``Approximation of linear forms'' is the general case.
\end{remark*}

\begin{remark*}[On attribution]
  When $(n,m) = (1,m)$ the above theorem is known as Khintchine's
  theorem~\cite[1924/1926]{Khintchineonedimensional,Khintchine}.  When
  $(n,m)$ is general, it is known as the Khintchine--Groshev
  theorem~\cite[1938]{Groshev}. An inhomogeneous version of
  Khintchine's theorem was proved by Sz{\"u}sz~\cite[1958]{SzuszinhomKT} in
  dimension $m=1$ and then Schmidt~\cite[1964]{Schmidt} in higher
  dimensions.  An inhomogeneous version of the Khintchine--Groshev
  theorem is found as~\cite[Theorem~12/15 in Chapter~1]{Sprindzuk} in Sprind\v{z}uk's book
  and it only requires monotonicity for $n=1, 2$.
\end{remark*}
 
The necessity of the emphasized monotonicity assumption
in~(\ref{eq:ikg}) has been one of the motivating questions of modern
research in Diophantine approximation, and there have been great
strides in the \emph{homogeneous} setting. Duffin and Schaeffer
constructed a counterexample showing that monotonicity cannot be
removed when $(n,m)=(1,1)$, and this gave birth to the
Duffin--Schaeffer conjecture~\cite[1941]{duffinschaeffer}, a problem
which invigorated the field for years until its eventual proof by
Koukoulopoulos and Maynard~\cite[2020]{KMDS}. Gallagher removed
monotonicity for $(1,m)$ when $m\geq 2$~\cite[1965]{Gallagherkt}. As
mentioned above, Sprind\v{z}uk proved Khintchine--Groshev without
monotonicity for $(n,m)$ when $n\geq 3$. Beresnevich and Velani
completed the homogeneous story by removing monotonicity when
$nm>1$~\cite[2010]{BVKG}, thus settling affirmatively a
  conjecture posed by Beresnevich, Bernik, Dodson, and Velani
  \cite[Conjecture A]{BBDV}. The work of Beresnevich and
  Velani shows that monotonicity can safely be removed in all
homogeneous cases except $(n,m)=(1,1)$, where the Duffin--Schaeffer
counterexample had already established that monotonicity could not be
removed.

The problem of removing monotonicity in the more general
\emph{inhomogeneous} part of the theory has lagged somewhat, with no
further progress having been recorded in the inhomogeneous
higher-dimensional linear forms setting since Sprind\v{z}uk proved the
general theorem without monotonicity for $(n,m)$ where $n\geq
3$. Aside from the earlier work of Sprind\v{z}uk, Yu recently removed
monotonicity from the general \emph{simultaneous} inhomogeneous
theorem when $n=1$ and $m\geq 3$~\cite[2021]{Yu}. Since homogeneous
approximation is a special case of inhomogeneous approximation,
corresponding to $\y = \0$, the Duffin--Schaeffer counterexample
already demonstrates that monotonicity cannot be removed in the
$(n,m)=(1,1)$ case. In fact, the second author showed that for
$(n,m)=(1,1)$ there is no inhomogeneous shift parameter $y \in \R$ for
which monotonicity can be removed~\cite[2017]{ds_counterex}.

This paper's main contribution to the inhomogeneous theory is to
remove monotonicity from the general Khintchine--Groshev theorem
whenever $nm>2$. This resolves the natural inhomogeneous
generalization of \cite[Conjecture A]{BBDV} when $nm>2$, leaving open
only the cases $(n,m) = (2,1)$ and $(n,m) = (1,2)$. We formally state
the two remaining cases of this inhomogeneous conjecture as
Conjecture~\ref{conj:nm=2} below, and we make partial progress towards
it in Theorem~\ref{thm:remaining}.

\begin{theorem}\label{thm:marquee}
  Let $nm>2$.
  Then for any
  $\psi:\NN\to \RR_{\geq 0}$ and $\y\in \RR^m$, we have
  \begin{equation*}
    \abs*{\calA_{n,m}^\y(\psi)} =
    \begin{cases}
      0 &\textrm{if } \sum_{q=1}^\infty q^{n-1}\psi(q)^m < \infty\;, \\[2ex]
      1 &\textrm{if } \sum_{q=1}^\infty q^{n-1}\psi(q)^m =\infty\;.
    \end{cases}
  \end{equation*}
\end{theorem}

\begin{remark*}
  As we have mentioned, the $n\geq 3$ parts of this theorem appear
  in~\cite{Sprindzuk}, and the $(1,m)$ cases with $m\geq 3$ appear
  in~\cite{Yu}. We present a conceptual proof that establishes the
  result for all $nm>2$.
\end{remark*}

Theorem~\ref{thm:marquee} is actually a special case of the following
more general theorem, which does not require the balls $(B_q)_{q \in \N}$ to be
concentric.

\begin{theorem}\label{thm:generalmarquee}
  Let $nm>2$.
  For any sequence of balls
  $\Psi :=(B_q)_{q \in \N}\subset\RR^m/\ZZ^m$, we have
  \begin{equation*}
    \abs*{\calA_{n,m}(\Psi)} =
    \begin{cases}
      0 &\textrm{if } \sum_{q=1}^\infty q^{n-1}\abs{B_q} < \infty\;, \\[2ex]
      1 &\textrm{if } \sum_{q=1}^\infty q^{n-1}\abs{B_q} =\infty\;.
    \end{cases}
  \end{equation*}
\end{theorem}

\begin{proof}[Proof of Theorem~\ref{thm:marquee}]
  This is the special case of Theorem~\ref{thm:generalmarquee} where
  all of the balls are concentric at the inhomogeneous parameter $\y$.
\end{proof}

Regarding the question of whether it is possible to remove the
monotonicity assumption from the general inhomogeneous
Khintchine--Groshev theorem, Theorem~\ref{thm:marquee} leaves open
only the cases $(1,2)$ and $(2,1)$. (Recall that we know the $(1,1)$
case to be impossible.) We remind the reader that monotonicity is not
required to prove the \emph{convergence} part of Theorem \ref{thm:marquee}
(nor is it required for the convergence part of Theorem
\ref{thm:generalmarquee}) in any case and, as stated above, the
convergence part of the general inhomogeneous Khintchine--Groshev
theorem is already known to be true without monotonicity for any
$n,m \geq 1$.  So, the question is really whether we can remove
monotonicity from the divergence part of the
statement. We suspect the truth of the following
  statement, which is the natural inhomogeneous analogue of
  \cite[Conjecture~A]{BBDV} in the cases when $nm=2$.

\begin{conjecture} \label{conj:nm=2}
Let $nm=2$. For any $\psi:\NN\to \RR_{\geq 0}$ and $\y\in \RR^m$, we have
  \begin{equation*}
    \abs*{\calA_{n,m}^\y(\psi)} = 1\qquad\textrm{if}\qquad \sum_{q=1}^\infty q^{n-1}\psi(q)^m =\infty.
  \end{equation*}
\end{conjecture}

Moreover, in line with the statement of Theorem \ref{thm:generalmarquee}, we actually expect the following more general result to be true when $nm=2$. Conjecture \ref{conj:nm=2} would follow from this more general result as an immediate corollary. 

\begin{conjecture}\label{conj:generalmarquee}
  Let $nm=2$. For any sequence of balls
  $\Psi :=(B_q)_{q \in \N}\subset\RR^m/\ZZ^m$, we have
  \begin{equation*}
    \abs*{\calA_{n,m}(\Psi)} = 1 \qquad\textrm{if}\qquad \sum_{q=1}^\infty q^{n-1}\abs{B_q} =\infty\;.
  \end{equation*}
\end{conjecture}

The challenge in the above conjectures is an issue that causes
difficulty in all problems of this type. Without going into detail, it
suffices to say that it comes from the fact that any rational point
can be expressed in infinitely many ways by changing
denominators. Authors contend with this issue in a variety of ways,
often under the banner of ``overlap estimates'' (see, for example, our
Lemma~\ref{lem:overlaps}). Indeed, the Duffin--Schaeffer conjecture
arose from an attempt to mitigate this difficulty by requiring that
all rational numbers be expressed in their reduced form, and adjusting
the divergence condition accordingly. Even after taking this measure,
the conjecture stood for almost 80 years. In the meantime, much of the
partial progress consisted of proofs of the conjecture under ``extra
divergence'' assumptions on the
series~\cite{AistleitnerDS,AistleitneretalExtraDivergence,BHHVextraii,HPVextra}. In
a similar vein, we have the following theorem, which can be seen as an
``extra divergence'' version of an inhomogeneous analogue of the
Duffin--Schaeffer conjecture for systems of linear forms.

\begin{theorem}\label{thm:remaining}
  Let $\eps>0$. For any $\psi:\NN\to \RR_{\geq 0}$ and $y\in \RR$, we
  have
  \begin{equation*}
    \abs*{\calA_{2,1}^y(\psi)} = 1\qquad\textrm{if}\qquad \sum_{q=1}^\infty \parens*{\frac{\varphi(q)}{q}}^{1+\eps} q \psi(q) =\infty.
  \end{equation*}
  For any
  $\psi:\NN\to \RR_{\geq 0}$ and $\y\in \RR^2$, we have
  \begin{equation*}
    \abs*{\calA_{1,2}^\y(\psi)} = 1\qquad\textrm{if}\qquad \sum_{q=1}^\infty \parens*{\frac{\varphi(q)}{q}}^{1+\eps} \psi(q)^2 =\infty.
  \end{equation*}
  Here, $\varphi$ denotes the Euler totient function.
\end{theorem}

\begin{remark*}
  This result improves on~\cite[Theorem~1.8]{Yu}, where
    the $(1,2)$ case is proved with an exponent $2$ in place of
    $1 + \eps$.
\end{remark*}

\begin{remark*} As with Theorem~\ref{thm:marquee},
    Theorem~\ref{thm:remaining} is a special case of a more general
    theorem that does not require concentric balls. It is listed here
    in Section \ref{sec:remaining} as
    Theorem~\ref{thm:generalextradiv}.
\end{remark*}

It is well-known that $\varphi(q) \gg q/\log\log q$. Therefore,
Theorem~\ref{thm:remaining} immediately implies that in the two cases
where $nm=2$, the divergence part of Theorem~\ref{thm:marquee} holds
under the extra divergence condition
\begin{equation*}
  \sum_{q=1}^\infty \frac{q^{n-1}\psi(q)^m}{(\log\log q)^{1+\eps}} =\infty.
\end{equation*}
Readers who are familiar with the Duffin--Schaeffer conjecture and its
``extra divergence'' precursors will naturally suspect that
Theorem~\ref{thm:remaining} is true even with $\eps=0$. Indeed,
Conjecture~\ref{conj:nm=2} already predicts something even
stronger. We therefore suggest that proving the ``$\eps=0$'' case of
Theorem~\ref{thm:remaining} may be seen as an interesting intermediate
challenge that is likely to be more tractable than full resolutions of
the above conjectures.

Interestingly, the guiding principle and key result of this paper
(presented in the next subsection lead us to believe that the
``$(n,m)=(2,1)$, $\eps=0$'' case of Theorem~\ref{thm:remaining} would
follow from a weak inhomogeneous version of the Duffin--Schaeffer
conjecture, stating that
  \[\abs*{\calA_{1,1}^{\y}(\psi)} = 1\qquad\textrm{if}\qquad
    \sum_{q=1}^{\infty}{\frac{\psi(q)\varphi(q)}{q}} = \infty,\]
  provided it were proved by establishing quasi-independence on
  average of the necessary sets. Informally, this guiding principle
  indicates that whenever one can prove something for $(n,m)$, one
  should expect to be able to do it for $(n+1,m)$ too.
  
  In the present paper, our main motivating objective has been to
  remove redundant monotonicity assumptions from the classical
  inhomogeneous Khintchine--Groshev theorem. In another direction,
  Dani, Laurent, and Nogueira \cite[2015]{DLN} proved refinements of
  the classical Khintchine--Groshev theorem where they imposed certain
  primitivity constraints on their ``approximating points''. The
  methods they used relied on the monotonicity of the approximating
  functions. In addition, their main result~\cite[Theorem~1.1]{DLN} is
  doubly metric, in the sense that it holds for almost every pair
  $(\bx, \by)\in \RR^{nm}\times\RR^m$. In future work we hope to investigate whether the
  approach presented in this paper may be adapted to remove
  monotonicity from statements given in \cite{DLN}, or to establish a
  singly metric version of their main result. Any progress in these
  directions would go some way towards addressing Problems 1 and 2
  posed by Laurent in~\cite[2016]{Laurent2016}.

\subsection{An ``independence inheritance'' theorem}
\label{sec:an-indep-inher}
The results described in this section are the main new tool for
establishing Theorems~\ref{thm:marquee} and~\ref{thm:generalmarquee}. They also provide an instructive point of
view on all statements of Khintchine--Groshev type, and so we find
them interesting in their own right. To put it informally, they tell
us that statements in the $(n,m)$ case imply statements in the
$(n+k, m)$ case, hence, we may take as a guiding principle that
\begin{equation*}
  \textrm{``Khintchine''}\qquad\implies\qquad\textrm{``Khintchine--Groshev''}
\end{equation*}
in all its various incarnations. In order to state this principle
precisely, some discussion is needed.

The set $\calA_{n,m}(\Psi)$ is the limsup set of a sequence of
measurable subsets of $\I^{nm}$. As such, its measure is $0$ whenever
the measures of those subsets have a converging sum --- a simple
consequence of the First Borel--Cantelli lemma (see, for example,
\cite[Lemma 1.2]{Harman}). Indeed, this is why, as we have mentioned,
the convergence parts of the theorems in the previous subsection are
straightforward to establish.

On the other hand, if the measure sum \emph{diverges}, it takes much
more work to show that $\calA_{n,m}(\Psi)$ has full, or even positive,
measure. (In fact, it might not have positive measure, as
counterexamples in the $(n,m)=(1,1)$ setting have shown.)  Invariably,
that work has involved showing that the sequence of sets exhibits some
sort of ``independence,'' such as is necessary to apply a partial
converse of the Borel--Cantelli lemma. In fact, Beresnevich and Velani
have shown that, in a sense made precise in~\cite{BVBC}, a limsup set
has positive measure \emph{only if} there is some independence
present.

For us, the relevant notion of independence is that of
quasi-independence on average. A sequence $(A_{\q})_{\q\in\ZZ^n}$ of
measurable subsets of a finite measure space $(X,\mu)$ is said to be
\emph{quasi-independent on average} (or QIA for short) if
\begin{equation}\label{eq:3}
  \limsup_{Q\to\infty} \frac{\parens*{\sum_{\abs{\q}=1}^Q \mu(A_\q)}^2}{\sum_{\abs{\q},\abs{\br}=1}^Q
  \mu(A_\q\cap A_\br)} > 0. 
\end{equation}
The importance of this concept is made clear by a well-known partial
converse to the first Borel--Cantelli lemma, stated here as
Lemma~\ref{lem:borelcantelli}. It guarantees that if a sequence is
quasi-independent on average, then the associated limsup set has
positive measure.

The sets we will be working with here are of the following form. Given
a ball $B\subset \RR^m/\ZZ^m$ and $\q \in \Z^{n}$, let $A_{n,m}(\q,B)$ denote the
set of $\x\in \I^{nm}$ for which there exists a $\p \in\ZZ^m$ with
$\q\x + \p \in B$. When the meaning is clear from context, we may omit
the subscript and simply write $A(\q,B)$.  When $n=1$ we customarily
write the non-bold $q$, i.e. $A(q,B)$. The sets $\calA_{n,m}(\Psi)$
appearing in our main results are limsup sets for sequences of sets of
this form, so the problem of establishing Theorems~\ref{thm:marquee}
and~\ref{thm:generalmarquee} naturally involves studying the
independence properties of the sets $A_{n,m}(\q, B)$.

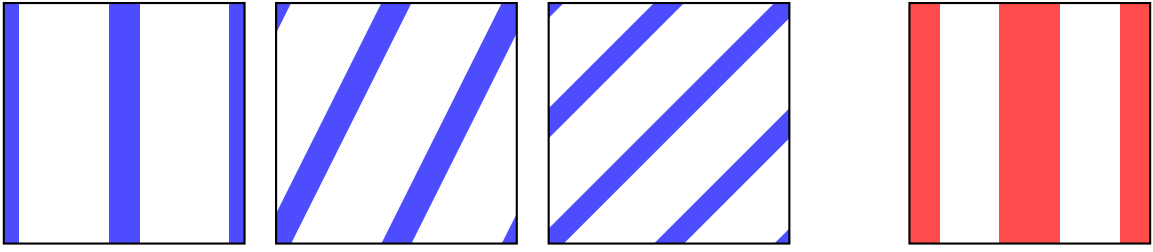
\begin{figure}[h]
  \centering
  \begin{framed}
      \begin{tikzpicture}[scale=0.4]
      \fill[fill opacity=0.7, color=blue] (0.5,0) -- (0.5,8)--(0,8)--(0,0)--(0.5, 0);

      \fill[fill opacity=0.7, color=blue] (4.5,0) --
      (4.5,8)--(3.5,8)--(3.5,0)--(4.5,0);

      \fill[fill opacity=0.7, color=blue] (8,0) -- (8,8)--(7.5,8)--(7.5,0)--(8, 0);

      \draw[thick] (0,0) rectangle (8,8);      
    \end{tikzpicture}\quad
    \begin{tikzpicture}[scale=0.4]

      \fill[fill opacity=0.7, color=blue] (0,7) -- (0.5,8)--(0,8)--(0,
      7);
      
      \fill[fill opacity=0.7, color=blue] (0.5,0) -- (4.5,8)--(3.5,8)--(0,
      1)--(0,0)--(0.5, 0);

     \fill[fill opacity=0.7, color=blue] (4.5,0) --
     (8,7)--(8,8)--(7.5,8)--(3.5, 0)--(4.5,0);

     \fill[fill opacity=0.7, color=blue] (7.5,0) --
     (8,0)--(8,1)--(7.5,0);

      \draw[thick] (0,0) rectangle (8,8);      
    \end{tikzpicture}\quad
    \begin{tikzpicture}[scale=0.4]

      \fill[fill opacity=0.7, color=blue] (0,7.5) -- (0.5,8)--(0,8)--(0,
      7.5);
      
      \fill[fill opacity=0.7, color=blue] (0,0) -- (0.5,0) -- (8,7.5) -- (8,8)--(7.5,
      8)--(0,0.5--(0,0);

     \fill[fill opacity=0.7,  color=blue] (4.5,0) --
     (8,3.5)--(8,4.5)--(3.5,0)--(4.5,0);

     \fill[fill opacity=0.7, color=blue] (7.5,0) --
     (8,0)--(8,0.5)--(7.5,0);

     \fill[fill opacity=0.7, color = blue] (0, 4.5) --
     (3.5,8)--(4.5,8)--(0,3.5)--(0,4.5);

      \draw[thick] (0,0) rectangle (8,8);      
    \end{tikzpicture}\qquad\qquad
    \begin{tikzpicture}[scale=0.4]
      \fill[fill opacity=0.7, color=red] (1,0) -- (1,8)--(0,8)--(0,0)--(1, 0);

      \fill[fill opacity=0.7, color=red] (5,0) --
      (5,8)--(3,8)--(3,0)--(5,0);

      \fill[fill opacity=0.7, color=red] (8,0) -- (8,8)--(7,8)--(7,0)--(8, 0);

      \draw[thick] (0,0) rectangle (8,8);      
    \end{tikzpicture}
    \caption{On the left are three of the $16$ possible sets
      $A_{2,1}(\q, B_q)$ with $\abs{\q}=q=2$. On the right is the set
      $A_{1,1}(\q, qB_q)\times [0,1]$. The areas of the $16$ possible
      blue regions sum to the area of the red region, up to a constant
      which depends on the dimensions; in this case, that constant is~$8$. One may imagine that the red region has been fractured into
      $q=2$ pieces of equal area, (duplicated $8$ times,) and shuffled
      in order to make the blue
      regions. Theorem~\ref{thm:stronginheritance} quantifies the
      extra independence that is gained in this fracturing and
      shuffling.}
  \label{fig:stripes}
  \end{framed}
\end{figure}

As a starting point, we take a na\"{\i}ve point of view, which we will
illustrate in the $(2,1)$ case (see Figure~\ref{fig:stripes}). The
divergence condition $\sum q\abs{B_q} = \infty$ should be understood as the
divergence of a measure sum. The summand $q\abs{B_q}$ is, up to a
constant, the sum of the measures of $A_{2,1}(\q, B_{\abs{\q}})$ as
$\q$ ranges through the sphere of radius $q$ in $\ZZ^2$. These sets
each have measure $\abs{B_q}$, and for each $\abs{\q}=q$, the set
$A_{2,1}(\q, B_q)$ can be visualized as an array of ``stripes''
through the unit square, at an angle determined by the angle of
$\q$. On the other hand one can interpret the sum $\sum q\abs{B_q}$ as
$\sum \abs{q B_q}$, where $qB_q$ denotes the $\times q$-dilation of
$B_q$ around its center. This way, one may imagine that the sum is the
measure sum corresponding to the sets $A_{1,1}(q, qB_q)$, each of
which is a union of intervals. Any metric properties of these sets
will be shared trivially by the vertically extended sets
$A_{1,1}(q,qB_q)\times [0,1]\subset \I^2$. Notice that this last set,
which is a union of thick vertical stripes, has the same measure (up
to a constant) as the \emph{sum} of the measures of the sets
$\set*{A_{2,1}(\q, B_q)}_{\abs{\q}=q}$, which are unions of thin
stripes arranged in many different directions. This leads to the na\"{\i}ve
suspicion that
\begin{equation*}
\textrm{\emph{the sets $\set*{A_{2,1}(\q, B_{\abs{\q}})}_{\q\in\ZZ^2}$ should be at least as independent as
  the sets $\set*{A_{1,1}(q,qB_q)}_{q\in\ZZ}$,}}
\end{equation*}
because it is as if the latter have each been fractured into $q$
pieces which are then shuffled into a disordered arrangement to create
the former.

The next result confirms this intuition (and a stronger result, to be
discussed, quantifies it). It states that quasi-independence on
average for certain sequences of sets in $\I^{nm}$ is inherited by an
associated sequence of sets in $\I^{(n+k)m}$ for any integer
$k \geq 0$, and this allows us to prove full measure for the
corresponding limsup set in $\I^{(n+k)m}$.

\begin{theorem}[Independence Inheritance]\label{thm:inheritance}
  Let $n,m\geq 1$. Suppose $(B_q)_{q\in\N}$ is a sequence of balls in
  $\RR^m/\ZZ^m$. If the sets
  $\parens*{A_{n,m}(\q, B_{\abs{\q}})}_{\q\in\ZZ^n}$ are quasi-independent
  on average, then so are the sets
  $\parens*{A_{n+k,m}(\q, \abs{\q}^{-k/m}B_{\abs{\q}})}_{\q\in\ZZ^{n+k}}$ for
  every $k\geq 0$.  If in addition the sum $\sum \abs{B_q}$ diverges,
  then for every $k\geq 0$, the set
  \begin{equation*}
    \limsup_{\abs{\q}\to\infty}A_{n+k,m}(\q, \abs{\q}^{-k/m}B_{\abs{\q}})
  \end{equation*}
  has full Lebesgue measure.
\end{theorem}

The philosophy behind applications of Theorem~\ref{thm:inheritance} is
that Khintchine--Groshev-like statements in the case $(n+k,m)$ are
weaker than in the case $(n,m)$. Indeed, by using 
  Theorem~\ref{thm:inheritance}, one can show that for $m\geq 2$, the $m$-dimensional
Khintchine theorem \emph{implies} the $(n,m)$-dimensional
Khintchine--Groshev theorem, because the sets involved in Khintchine's
theorem are quasi-independent on average in dimensions
$\geq 2$~\cite{Gallagherkt}. The same implication holds
inhomogeneously, as well as inhomogeneously \emph{sans} monotonicity
in dimension $\geq 3$; that is, the inhomogeneous Khintchine theorem
in dimension $m\geq 3$ implies the inhomogeneous Khintchine--Groshev
theorem in dimensions $(n,m)$ for any $n$. The quasi-independence on
average of the relevant sets in the cases $(1,m)$ with $m\geq 3$ is
established in Proposition~\ref{prop:basecase} as well as in Yu's
proof of~\cite[Theorem~1.8]{Yu}.

In fact, Theorem~\ref{thm:inheritance} follows from a stronger
inheritance phenomenon, Theorem~\ref{thm:stronginheritance}, which
allows us to quantify how much ``more independent'' the sets in the
$(n+k,m)$ case are than the sets in the $(n,m)$ case. We define in
Section~\ref{sec:proof-theor-refthm:m} (Definition~\ref{def:weakind})
a hierachy of the form
\begin{equation*}
  \textrm{QIA} = 0\textrm{-QIA} \quad\implies\quad 1\textrm{-QIA} \quad\implies\quad 2\textrm{-QIA}\quad\implies\quad 3\textrm{-QIA}\quad\implies\quad\cdots
\end{equation*}
where $w$-QIA stands for ``$w$-weak quasi-independence on average''
and $0$-QIA coincides with the usual notion defined
by~(\ref{eq:3}). Theorem~\ref{thm:stronginheritance} states that the
$(n+k,m)$ case inherits an independence that is $k$ steps stronger
than the independence enjoyed by the $(n,m)$ case. At this level of
detail, we can already say how Theorems~\ref{thm:marquee}
and~\ref{thm:generalmarquee} are proved: 
we establish that the sets involved in the $(1,1)$ case are
$2$-QIA, the sets in the $(1,2)$ case are $1$-QIA, and the sets in the
$(1,m)$ case ($m\geq 3$) are $0$-QIA (see
Proposition~\ref{prop:basecase}).  Theorem~\ref{thm:stronginheritance}
then gives us $0$-QIA for all $(n,m)$ with $nm>2$.

The full measure statement then follows by a ``mixing'' argument
exploiting the periodicity inherent in the sets $A_{n,m}(\q, B)$ (see
Lemma~\ref{lem:mixing} and Proposition~\ref{prop:fullmeasure}). This
argument enables us to overcome another perceived obstacle in our
general setting; namely, that the literature lacks a
``zero-one'' law for inhomogeneous Diophantine approximation. In the
homogeneous setting, Beresnevich and Velani  had
  previously shown that the Lebesgue measure of the set
$\cA_{n,m}^{\0}(\psi)$ is always either zero or one; that is, for
homogeneous approximation, these sets satisfy a so-called ``zero-one''
law \cite{BVzero-one}. A consequence of this was that in proving the
homogeneous Khintchine--Groshev theorem without monotonicity for
$nm>1$ in \cite{BVKG}, to obtain the full measure statement, it was
enough for Beresnevich and Velani to show that the sets
$\cA_{n,m}^{\0}(\psi)$ had \emph{positive} measure subject to the
divergence of the appropriate sum. This is a common and effective
strategy that has been used countless times in this
field. Consequently, zero-one laws are seen as indispensable tools in
homogeneous approximation, the most well-known of them being the
zero-one laws of Cassels and
Gallagher~\cite{Casselslemma,Gallagher01}. Since we know
  of no zero-one law for the general sets $\cA_{n,m}^{\y}(\psi)$, that
  strategy is unavailable to us. However, our
  Theorem~\ref{thm:marquee} implies an \emph{a posteriori} inhomogeneous zero-one
  law when $nm>2$.

\subsection{The Hausdorff measure theory} \label{sec:Hausdorff}
 
In addition to studying the Lebesgue measure of sets such as
$\calA_{n,m}^\y(\psi)$, much attention in Diophantine approximation is
also devoted to studying the more refined Hausdorff measures and
Hausdorff dimension of such sets. This is particularly fruitful when
the Lebesgue measure of such sets is zero, in which case Hausdorff
dimension and Hausdorff measures can often still provide us with a
means of differentiating the relative ``sizes'' of such sets. In
classical simultaneous approximation, these types of distinctions were
accomplished by the Jarn\'{\i}k--Besicovitch theorem and Jarn\'{\i}k's
theorem. See~\cite{beresnevich_ramirez_velani_2016} for a survey
discussion.  

We conclude this introduction by recording Hausdorff measure
counterparts to Theorems~\ref{thm:marquee}, 
\ref{thm:remaining},  and \ref{thm:generalmarquee}. These are, respectively, Theorems \ref{thm: Hausdorff marquee}, \ref{thm: Hausdorff assuming}, and \ref{thm: general Hausdorff marquee}. Theorems~\ref{thm: Hausdorff marquee} and
\ref{thm: Hausdorff assuming} can be obtained as a simple consequence of
combining the relevant Lebesgue measure statement (in our case,
Theorem \ref{thm:marquee} or \ref{thm:remaining}) with \cite[Theorem
5]{AllenBeresnevich}. The latter statement, appearing below as Theorem
AB, is essentially a general ``transference principle,'' proved by the
first author and Beresnevich, which allows us to more-or-less
immediately read off a Hausdorff measure analogue when given an
appropriate Lebesgue measure Khintchine--Groshev type
statement. Theorem AB is a consequence of the mass transference
principle for systems of linear forms \cite[Theorem
1]{AllenBeresnevich} which, informally speaking, allows us to transfer
Lebesgue measure statements for limsup sets determined by
neighbourhoods of planes (i.e. systems of linear forms) to Hausdorff
measure statements.  The mass transference principle for systems of
linear forms is just one natural generalization of the original
mass transference principle due to Beresnevich and Velani
\cite[Theorem~2]{BVMTP}.

\begin{theoremAB*}[\cite{AllenBeresnevich}] \label{thm:Lebesgue to
    Hausdorff transference} Let $\psi: \N \to \R_{\geq 0}$ be an
  approximating function, let $\y \in \I^m$, and let $f$ and
  $g: r \to g(r) := r^{-m(n-1)}f(r)$ be dimension functions such that
  $r^{-nm}f(r)$ is monotonic. Let
$$
\theta: \N \to \R_{\geq 0} \qquad\text{be defined by}\qquad \theta(q)=
q\,g\left(\frac{\psi(q)}{q}\right)^{\frac{1}{m}}\,.
$$
Then
$$
|\cA_{n,m}^{\y}(\theta)| = 1\qquad\text{implies}\qquad
\cH^f(\cA_{n,m}^{\y}(\psi)) = \cH^f(\I^{nm}).
$$
\end{theoremAB*}

Before presenting our Hausdorff measure results, we recall some
definitions. We will say that a function $f: \R_{>0} \to \R_{>0}$ is a
\emph{dimension function} if $f$ is continuous, non-decreasing, and
$f(r) \to 0$ as $r \to 0$. Given a set $F \subset \R^k$ and a real
number $\rho > 0$, we will say that a countable collection of balls,
say $(B_i:=B(x_i,r_i))_{i \in \N}$, is a \emph{$\rho$-cover} for $F$
if $F \subset \bigcup_{i=1}^{\infty}{B_i}$ and $r_i<\rho$ for all
$i \in \N$. The \emph{$\rho$-approximate Hausdorff $f$-measure} of $F$
is defined as
\[\cH_{\rho}^{f}(F):= \inf\left\{\sum_{i=1}^{\infty}{f(r)}: \{B_i\}_{i \in \N} \text{ is a $\rho$-cover for $F$}\right\}.\]
The \emph{Hausdorff $f$-measure} of $F$ is then defined to be
\[\cH^{f}(F):= \lim_{\rho \to 0}{\cH_{\rho}^{f}(F)} = \sup_{\rho > 0}{\cH_{\rho}^{f}(F)}.\]
In the case that $f(r)=r^s$ for some real $s \geq 0$, $\cH^f$ is the
perhaps more familiar \emph{Hausdorff $s$-measure}. In this case we
write $\cH^s$ in place of $\cH^f$. Although we will not require this
definition here, for completeness we remark that the \emph{Hausdorff
  dimension} of $F$ is defined as
\[\dimh{F} := \inf\set{s \geq 0: \cH^s(F) = 0}.\] 
For further details on Hausdorff measures and Hausdorff dimension, we refer the reader to, for example, \cite{Falconer, Mattila, Rogers}. 

Combining Theorem \ref{thm:marquee} with Theorem AB yields the
following Hausdorff measure analogue of Theorem
\ref{thm:marquee}. This statement is essentially an inhomogeneous \emph{Hausdorff
  measure Khintchine--Groshev theorem}.

\begin{theorem} \label{thm: Hausdorff marquee} 
  Let $nm>2$. Let $\y \in \I^{m}$, and let
  $\psi: \N \to \R_{\geq 0}$ be an approximating function. Let $f$ and
  $g: r \to g(r):=r^{-m(n-1)}f(r)$ be dimension functions such that
  $r^{-nm}f(r)$ is monotonic. Then,
$$
\cH^f(\cA_{n,m}^{\y}(\psi)) =\left\{
\begin{array}{lcl}
 \displaystyle 0 \quad & \mbox{ \text{if}\quad
 $\sum_{q=1}^{\infty}{q^{n+m-1}g\left(\frac{\psi(q)}{q}\right)} < \infty$\;,}\\[2ex]
 \displaystyle \cH^f(\mathbb{I}^{nm}) \quad & \mbox{ \text{if}\quad
 $\sum_{q=1}^{\infty}{q^{n+m-1}g\left(\frac{\psi(q)}{q}\right)} = \infty$\;.}
\end{array}
\right.
$$
\end{theorem}

\begin{proof}
Analogously to the Lebesgue measure theory, the convergence part of Theorem \ref{thm: Hausdorff marquee} follows from a standard covering argument, and the statement is true in this case for any $n,m \geq 1$ with no monotonicity requirements on $\psi$. We omit the details here but note that the required argument is given very explicitly in \cite{AllenThesis}. 

The divergence part of the above statement would follow immediately
from Theorem~AB provided that we could show that
$\abs{\cA_{n,m}^{\y}(\theta)}=1$ for
$\theta(q)=qg\left(\frac{\psi(q)}{q}\right)^{\frac{1}{m}}$. Provided
that $nm>2$, it follows from
Theorem \ref{thm:marquee} that $\abs{\cA_{n,m}^{\y}(\theta)}=1$ if
\[\sum_{q=1}^{\infty}{q^{n-1}\theta(q)^m} = \sum_{q=1}^{\infty}{q^{n+m-1}g\left(\frac{\psi(q)}{q}\right)} = \infty,\]
thus completing the proof.
\end{proof}

\begin{remark*}
  The proof of Theorem \ref{thm: Hausdorff marquee}, both the
  convergence and divergence parts, is virtually identical to the
  proof given of Theorem 3.7 in \cite{AllenThesis}.  In all cases
  except when $n=1$ or $n=2$, Theorem \ref{thm: Hausdorff marquee} was
  already known to be true without any monotonicity assumptions on
  $\psi$, see \cite[Theorem 3]{AllenBeresnevich} (which also appears
  as \cite[Theorem 3.7]{AllenThesis}). In the case that $n=1$ or
  $n=2$, the conclusion of Theorem \ref{thm: Hausdorff marquee} was
  shown to be true in \cite[Theorem 3]{AllenBeresnevich} subject to
  additional monotonicity assumptions.  We refer the reader to
  \cite{AllenThesis, AllenBeresnevich} for detailed discussion of the
  precise requirements, as well as for more details on the mass
  transference principle for linear forms.
\end{remark*}

In line with how Theorem~\ref{thm:marquee} demonstrates that
monotonicity is not required in the classical inhomogeneous
Khintchine--Groshev theorem whenever $nm>2$, Theorem \ref{thm: Hausdorff marquee} shows that the same is true in
the analogous inhomogeneous Hausdorff measure Khintchine--Groshev theorem. Moreover, any further progress towards removing monotonicity from the Lebesgue measure inhomogeneous Khintchine--Groshev theorem in the remaining $(1,2)$ and $(2,1)$ cases
can be transferred to analogous progress in the
Hausdorff measure setting via Theorem AB.  For example, in the same
way that we can deduce the divergence part of Theorem \ref{thm:
  Hausdorff marquee} from Theorem \ref{thm:marquee} using Theorem AB,
we can deduce the following Hausdorff measure analogue of
Theorem~\ref{thm:remaining}. We leave the details as an exercise for
the interested reader!

\begin{theorem}\label{thm: Hausdorff assuming}
  Let $\varepsilon > 0$ and let $nm=2$. Let $f$ and
  $g: r \to g(r):=r^{-m(n-1)}f(r)$ be dimension functions such that
  $r^{-nm}f(r)$ is monotonic.
 Then, for any $\psi: \N \to \R_{\geq 0}$, and any $\y \in \I^m$, we have
  \begin{equation*}
    \cH^f\left(\cA_{n,m}^{\y}(\psi)\right) = \cH^f(\I^{nm}) \qquad\textrm{if}\qquad \sum_{q=1}^{\infty}{\left(\frac{\varphi(q)}{q}\right)^{1+\varepsilon} q^{n+m-1} g\left(\frac{\psi(q)}{q}\right)} = \infty.
  \end{equation*}
\end{theorem}

\begin{remark*}
  Notice that the condition $nm=2$ only permits the cases $(1,2)$ and
  $(2,1)$.
\end{remark*}

Via a more direct application of the mass transference principle for
systems of linear forms \cite[Theorem 1]{AllenBeresnevich}, it is also
possible to prove a more general Hausdorff measure statement analogous
to the Lebesgue measure statement given by Theorem
\ref{thm:generalmarquee}. More precisely, recall that if
$\Psi:= (B_q)_{q \in \N}$ is a sequence of balls in $\R^m/\Z^m$ and
$n \geq 1$, we denote by $\cA_{n,m}(\Psi)$ the set of $\x \in \I^{nm}$
such that
\[\q\x+\p \in B_{|\q|}\] for infinitely many pairs
$(\p,\q) \in \Z^m \times \Z^n$. By adapting the proof
of \cite[Theorem~2]{AllenBeresnevich} (specifically the proof of (11)
in \cite{AllenBeresnevich}), one can apply \cite[Theorem
1]{AllenBeresnevich} directly to obtain the following Hausdorff
measure analogue of Theorem~\ref{thm:generalmarquee}.

\begin{theorem}\label{thm: general Hausdorff marquee}
Let $nm > 2$.
Let $f$ and $g: r \to g(r) = r^{-m(n-1)}f(r)$ be dimension functions such that $r^{-nm}f(r)$ is monotonic. Then,
$$
\cH^f(\cA_{n,m}(\Psi)) =\left\{
\begin{array}{lcl}
 \displaystyle 0 \quad & \mbox{ \text{if}\quad
 $\sum_{q=1}^{\infty}{q^{n+m-1}g\left(\frac{|B_q|^{\frac{1}{m}}}{2q}\right)} < \infty$,}\\[2ex]
 \displaystyle \cH^f(\mathbb{I}^{nm}) \quad & \mbox{ \text{if}\quad
 $\sum_{q=1}^{\infty}{q^{n+m-1}g\left(\frac{|B_q|^{\frac{1}{m}}}{2q}\right)} = \infty$.}
\end{array}
\right.
$$
\end{theorem}

\begin{remark*}
Observe that $\frac{|B_q|^{\frac{1}{m}}}{2q}$ is simply $\frac{r(B_q)}{q}$, where $r(B_q)$ is the radius of the ball $B_q$.
\end{remark*}

\section{Lemmas}\label{sec:lemmas}

The first three lemmas are standard facts from measure theory. Recall
that the \emph{limsup} of a sequence of measurable sets
$(A_q)_{q \in \N}$ in a finite measure space $(X, \mu)$ is
\begin{align*}
\limsup_{q \to \infty}{A_q} :&= \left\{x \in X: x \in A_q \text{ for infinitely many $q \in \N$}\right\} \\
                             &= \bigcap_{Q \geq 1}{\bigcup_{q \geq Q}{A_q}}.
\end{align*}

\begin{lemma}[Divergence Borel--Cantelli Lemma,~{\cite[Lemma 2.3]{Harman}}]\label{lem:borelcantelli}
  Suppose $(X, \mu)$ is a finite measure space and $(A_q)_{q \in \N} \subset X$ is a
  sequence of measurable subsets such that $\sum \mu(A_q)=\infty$. Then
  \begin{equation*}
    \mu\parens*{\limsup_{q\to\infty} A_q} \geq \limsup_{Q\to\infty} \frac{\parens*{\sum_{q=1}^Q \mu(A_q)}^2}{\sum_{q,r=1}^Q\mu(A_q\cap A_r)}.
  \end{equation*}
\end{lemma}

\begin{lemma}[Chung--Erd\H{o}s Lemma,~{\cite[Equation 4]{ChungErdos}}]\label{lem:erdoschung}
  Suppose $(X, \mu)$ is a probability space and $(A_q)_{q \in \N}\subset X$ is a
  sequence of measurable subsets. If $\mu\parens*{\bigcup_{q=1}^Q  A_q}>0$, then
  \begin{equation*}
    \mu\parens*{\bigcup_{q=1}^Q  A_q} \geq \frac{\parens*{\sum_{q=1}^Q \mu(A_q)}^2}{\sum_{q,r=1}^Q\mu(A_q\cap A_r)}.
  \end{equation*}
\end{lemma}

\begin{lemma}[{\cite[Lemma 6]{BDV}}]\label{lem:lebesguedensity}
  Let $(X,d)$ be a metric space with a finite measure $\mu$ such that
  every open set is $\mu$-measurable. Let $A$ be a Borel subset of $X$
  and let $f:\RR_{\geq 0}\to\RR_{\geq 0}$ be an increasing function with $f(x)\to 0$ as
  $x\to 0$. If for every open set $U\subset X$ we have
  \begin{equation*}
    \mu(A\cap U) \geq f(\mu(U)),
  \end{equation*}
  then $\mu(A) = \mu(X)$.
\end{lemma}

The next lemmas are specific to our problem. Recall that for $\q\in \ZZ^n$ and a ball $B \subset \RR^m/\ZZ^m$ we define
\begin{equation}\label{eq:defB}
 A(\q, B) = A_{n,m}(\q, B) = \set{\bx\in\I^{nm} :  \q\bx + \p \in B \text{ for some } \p \in \Z^m}
\end{equation}
and note that it has measure $\abs{B}$. Given two vectors $\q_1, \q_2 \in \Z^n$, we will use the notation $\q_1 \parallel \q_2$ to denote that $\q_1$ and $\q_2$ are parallel and, similarly, we will write $\q_1 \nparallel \q_2$ to indicate that the vectors $\q_1$ and $\q_2$ are not parallel.

\begin{lemma}[{\cite[Lemma~9 in Chapter 1]{Sprindzuk}}]\label{lem:nparallel}
  Suppose $B_1, B_2\subset \RR^m/\ZZ^m$ are balls, and $\q_1,\q_2\in\ZZ^n$ with $\q_1\nparallel\q_2$. Then
    \begin{equation*}
      \abs*{A(\q_1, B_1)\cap A(\q_2, B_2)} = \abs*{A(\q_1, B_1)}\abs*{A(\q_2, B_2)}.
  \end{equation*}
\end{lemma}

\begin{lemma}\label{lem:invariance}
  Suppose $B_1, B_2 \subset \RR^m/\ZZ^m$ are balls, and that
  $\q_1\parallel\q_2\in\ZZ^n$ and $\br_1\parallel\br_2\in\ZZ^{n'}$ are
  such that $\abs{\q_i} = \abs{\br_i}$ and
  $\q_1\cdot\q_2 = \br_1\cdot \br_2$ (that is, the pairs either both
  point in the same direction, or both point in opposite
  directions). Then
  \begin{equation*}
\abs*{A_{n,m}(\q_1, B_1)\cap A_{n,m}(\q_2, B_2)} =\abs*{A_{n',m}(\br_1, B_1)\cap A_{n',m}(\br_2, B_2)}. 
\end{equation*}
\end{lemma}

\begin{remark*}
  Throughout this work, we have allowed $\q$ to take values in
  $\ZZ^n$. As it turns out, all of our arguments work if we restrict
  $\q$ to lie in the positive orthant $\ZZ^n_{\geq 0}$. Under that restriction, the question (seen for example in
  the above lemma) of whether $\q_1$ and $\q_2$ point in the same
  direction becomes irrelevant.
\end{remark*}

\begin{proof}[Proof of Lemma \ref{lem:invariance}]
We will show that 
\begin{equation}\label{eq:projection}
\abs*{A_{n,m}(\q_1, B_1)\cap A_{n,m}(\q_2, B_2)} =\abs*{A_{1,m}(\abs{\q_1}, B_1)\cap A_{1,m}(\pm\abs{\q_2}, B_2)},
\end{equation}
where the $\pm$ is determined according to whether $\q_1\cdot\q_2$ is
positive or negative. To see this, let $\q\in\ZZ^n$ and
$k_1,k_2\in\ZZ$ be such that $\q_1 = k_1\q$ and $\q_2 = k_2\q$. By
possibly changing the sign on $\q$ we may ensure that $k_1$ is
positive.  Now $A(\q_1, B_1)\cap A(\q_2, B_2)$ is the preimage of
$A(k_1, B_1)\cap A(k_2, B_2)$ under the projection
$\x\mapsto \q\x (\bmod 1)$. This mapping is measure-preserving
(see~\cite[Lemma~8 in Chapter 1]{Sprindzuk}), therefore
\begin{equation*}
\abs*{A(\q_1, B_1)\cap A(\q_2, B_2)} =\abs*{A(k_1, B_1)\cap A(k_2, B_2)}. 
\end{equation*}
But notice that $A(\abs{\q_1}, B_1)\cap A(\sgn(k_2)\abs{\q_2}, B_2)$ is the preimage of
$A(k_1, B_1)\cap A(k_2, B_2)$ under the ``$\times \abs{\q} (\bmod 1)$''
map, which is also measure-preserving, so we have
\begin{equation*}
\abs*{A(k_1, B_1)\cap A(k_2, B_2)} = \abs*{A(\abs{\q_1}, B_1)\cap A(\sgn(k_2)\abs{\q_2}, B_2)}.
\end{equation*}
Combining this with the previous observation, we
have~(\ref{eq:projection}), where the $\pm$ is determined
  by $\sgn(k_2)$, which is itself determined by whether $\q_1$ and
  $\q_2$ point in the same or opposite direction. Since the same
argument could have been carried out with $\br_1$ and $\br_2$, we are
done.
\end{proof}

\begin{lemma}\label{lem:dilate}
  Suppose $B_1, B_2 \subset \RR^m/\ZZ^m$ are balls, and $q_1, q_2$ are
  integers with $1\leq q_1 \leq q_2$. Then, for every $\q_1,\q_2\in\ZZ^n$ with $\q_1\parallel\q_2$ and
  $\abs{\q_i}=q_i$, 
    \begin{equation*}
\abs*{A(\q_1, q_1^{-1/m}B_1)\cap A(\q_2, q_2^{-1/m}B_2)} \leq  \frac{1}{q_1} \abs*{A(\q_1, B_1)\cap A(\q_2, (q_1/q_2)^{1/m}B_2)}.
  \end{equation*}
\end{lemma}

\begin{proof}
  Lemma~\ref{lem:invariance} implies that it is enough to prove this
  for $n=1$, that is, we must prove that
  \begin{equation*}
    \abs*{A(q_1, q_1^{-1/m}B_1)\cap A(\pm q_2, q_2^{-1/m}B_2)} \leq  \frac{1}{q_1} \abs*{A(q_1, B_1)\cap A(\pm q_2, (q_1/q_2)^{1/m}B_2)},
  \end{equation*}
  where the $\pm$ is determined by whether $\q_1$ and $\q_2$ point in
  the same or opposite direction. Seeing
  $A(q_1, B_1)\cap A(\pm q_2, (q_1/q_2)^{1/m} B_2)$ as a subset of
  $[0,1]^m$, we contract it by $q_1^{-1/m}$ and see that
  \begin{equation*}
  \abs*{A(q_1,B_1)\cap A(\pm q_2, (q_1/q_2)^{1/m}B_2)} = q_1\abs*{q_1^{-1/m}\parens*{A(q_1,  B_1)\cap A(\pm q_2,(q_1/q_2)^{1/m}B_2)}}. 
\end{equation*}
Notice that
\begin{equation*}
 q_1^{-1/m}\parens*{A(q_1,  B_1)\cap A(\pm q_2,(q_1/q_2)^{1/m} B_2)}
  =\underbrace{\brackets*{q_1^{-1/m}\parens*{A(q_1,B_1)}}}_C\cap \underbrace{\brackets*{q_1^{-1/m}\parens*{A(\pm q_2,(q_1/q_2)^{1/m}B_2)}}}_D.
\end{equation*}
Since $B_1$ and $B_2$ are balls in $\RR^m/\ZZ^m$, they have diameter
$\leq 1$, and so each of the square bracketed sets on the right-hand
side (labelled $C$ and $D$) is a union of essentially disjoint balls
of volume $q_1^{-1}\abs{B_1}$ and $q_2^{-1}\abs{B_2}$ respectively. We
may see $C\cap D$ as an intersection of several balls in
$(q_1^{-1/m}\RR)^m/(q_1^{-1/m}\ZZ)^m$ (a shrunken torus) such that no
point is contained in more than two of them. The sets
$A(q_1, q_1^{-1/m}B_1)$ and $A(\pm q_2, q_2^{-1/m}B_2)$ are obtained
from $C$ and $D$ by \emph{translating} the disjoint balls so that
their centers lie at their corresponding points in
$\RR^m/\ZZ^m$. Since those corresponding points can be obtained via a
scaling of the metric by $q_1^{1/m}\geq 1$, the pairwise distances
between centers has necessarily increased. Moreover, since the balls
making up $C$ and $D$ are only \emph{translated} to their new
positions (not scaled), the measures of the pairwise intersections
cannot have increased. Therefore, we have
\begin{equation*}
  \abs*{A(q_1, B_1)\cap A(\pm q_2, (q_1/q_2)^{1/m}B_2)} \geq q_1 \abs*{A(q_1, q_1^{-1/m}B_1)\cap A(\pm q_2, q_2^{-1/m}B_2)},
\end{equation*}
proving the lemma.
\end{proof}

\begin{lemma}\label{lem:mixing}
  For any $n,m \geq 1$ there exists a constant $C:=C_m>0$ such that for every open set $U\subset \I^{nm}$ the following holds: for all $\q \in \Z^n$ of sufficiently large norm ($\abs{\q}\geq Q_U$),
  \begin{equation*}
    \abs*{A(\q, B) \cap U} \geq C \abs*{A(\q,B)}\abs*{U}
  \end{equation*}
  holds for every ball $B\subset \RR^m/\ZZ^m$. (Note that the subscripts are added to the constants here to indicate their dependencies, e.g. $C_m$ indicates that $C$ depends on $m$.)
\end{lemma}

\begin{proof}
  First, find within $U$ a finite union of disjoint balls $V\subset U$
  such that $\abs{V}\geq \abs{U}/2$. (In principle, we could get as
  close to the measure of $U$ as we want.) We may even, and indeed
  will, assume all the balls in $V$ have the same radius, $r>0$.

  Now let $W\subset \I^{nm}$ be any ball of radius $r$. It will be
  enough to show that
  \begin{equation} \label{eq:mixingballssubcase}
    \abs*{A(\q, B) \cap W} \geq C' \abs*{A(\q,B)}\abs*{W}
  \end{equation}
  for all $\abs{\q}\geq Q_r$, where $C'>0$ is some absolute constant
  which may depend on $n$ and $m$, and $Q_r$ only depends on
  $r$. Importantly for us, $Q_r$ does not depend on
  $B$ and so, given \eqref{eq:mixingballssubcase}, one can deduce the lemma with $C = C'/2$.

  Let us write $\x \in \I^{nm}$ as $\x = (\x_1, \x_2,\cdots, \x_m)$
  where $\x_j$ are column vectors. Then for any $\q\in\ZZ^n$ we have
  that $\q\x = (\q\cdot\x_1, \dots, \q\cdot \x_m) \in \I^m$, and the
  condition that
  \begin{equation*}
    \q\x + \p\in B\qquad\textrm{for}\qquad\p=(p_1, \dots, p_m)\in\ZZ^m
  \end{equation*}
  is equivalent to $\q\cdot \x_j + p_j \in B_j$ for each $j$, where
  $B_j$ is the projection of $B$ to the $j$th component. Therefore, we
  have
  \begin{equation}\label{eq:product}
    A_{n,m}(\q,B)\cap W = \parens*{\prod_{j=1}^m A_{n,1}(\q, B_j)}\cap W =  \prod_{j=1}^m A_{n,1}(\q, B_j)\cap W_j, 
  \end{equation}
  where $W_j$ is the projection of $W$ to the copy of $\I^n$
  corresponding to the $j$th column of $\I^{nm}$.

  Suppose for the moment that $\abs{\q}$ is achieved in the first
  coordinate of $\q$. For any $\bz\in\I^{n-1}$ (representing the last
  $n-1$ coordinates), let
  \begin{equation*}
    S_\bz = \parens*{A_{n,1}(\q, B_j)\cap W_j}_\bz = A_{n,1}(\q, B_j)_\bz\cap (W_j)_\bz
  \end{equation*}
  be the cross-section through $\bz$ parallel to the first
  coordinate. Then 
  \begin{align*}
    \abs*{A_{n,1}(\q, B_j)\cap W_j} &= \int_{\I^{n-1}}\abs{S_\bz}\,d\bz \\
    &= \int_{Y_j}\abs{S_\bz}\,d\bz
  \end{align*}
  where $Y_j$ is the projection of $W_j$ to the last $n-1$
  coordinates. Meanwhile, $(W_j)_\bz$ is an interval of length $2r$
  and $A_{n,1}(\q, B_j)_\bz$ is a union of $\abs{\q}$ disjoint
  intervals of length
  $\abs{B_j}/\abs{\q} =
  \abs{B}^{1/m}/\abs{\q}$ with centers spaced
  $1/\abs{\q}$ apart, in $(\I^{nm})_{\bz} \cong \RR/\ZZ$. Therefore,
  $(W_j)_{\bz}$ fully contains at least \mbox{$2r\abs{\q} - 2$} of the
  intervals constituting $A_{n,1}(\q, B_j)_\bz$, so we have
  \begin{equation*}
   \abs{S_\bz} \geq (2r\abs{\q} - 2)\frac{\abs{B_j}}{\abs{\q}},
  \end{equation*}
  which exceeds $r\abs{B_j}$ as soon as $\abs{\q}\geq 2/r$. In that case, we have
  \begin{align*}
    \abs*{A_{n,1}(\q, B_j)\cap W_j} &= \int_{Y_j}\abs{S_\bz}\,d\bz \\
                                    &\geq r\abs{B_j}\int_{Y_j} d\bz \\
                                    &= r\abs{B_j}\abs{Y_j} \\
    &= 2^{n-1}r^n\abs{B_j}.
  \end{align*}
  Note that the argument in this paragraph did not depend on the
  supposition that $\abs{\q}$ was achieved in the first of the $n$
  coordinates.
  
  Now, by~(\ref{eq:product}) and recalling that
  $W \subset \I^{nm}$ was an arbitrary ball of radius $r$, we have
  \begin{equation*}
    \abs*{A(\q,B)\cap W} \geq 2^{nm-m}r^{nm}\prod_{j=1}^m\abs{B_j} \geq \frac{1}{2^{m}}\abs{B}\abs{W} 
  \end{equation*}
  for all $\abs{\q}\geq 2/r$. Since $\abs{A(\q, B)} = \abs{B}$, we are
  done. 
\end{proof}

\begin{lemma}[Overlap estimates]\label{lem:overlaps}
  Let $m\geq 1$. Suppose $B_1$ and $B_2$ are balls in $\RR^m/\ZZ^m$ and let $r,q \in \ZZ\setminus\set{0}$. Then
\begin{equation*}
  \abs*{A_{1,m}(r,B_1)\cap A_{1,m}(q,B_2)} \ll \abs{B_1}\abs{B_2} +  \abs{B_2}\abs{q}^{-m}\gcd(r,q)^m,
\end{equation*}
where the implicit constant depends only on $m$. In fact, we can take
the implicit constant to be $2^m$.
\end{lemma}

\begin{proof}
Since $A(q,B) = A(-q, -B)$, it is enough to prove this lemma for $r,q\geq 1$.

  Suppose the radii of $B_1$ and $B_2$ are
  $\psi_1$ and $\psi_2$, respectively. Let
  \[\delta = 2\min\left\{\frac{\psi_1}{r}, \frac{\psi_2}{q}\right\} \qquad \text{and} \qquad \Delta = 2\max\left\{\frac{\psi_1}{r}, \frac{\psi_2}{q}\right\}.\] 
  Then
\begin{equation*}
  \abs*{A(r,B_1)\cap A(q,B_2)} \leq \delta^m N
\end{equation*}
where
\begin{equation*}
N = \#\set*{(\mathbf a,\mathbf b)\in \ZZ^m\times\ZZ^m : \abs{\mathbf a} \leq q, \abs{\mathbf b}\leq r,  r\mathbf a - q\mathbf b \in B},
\end{equation*}
and $B$ is a ball of diameter $\Delta rq$ whose center depends on
$B_1, B_2, r,q$, but we do not need to specify it here. Essentially,
$N$ gives us an upper bound on the number of regions of intersection
we can possibly see between $A(r,B_1)$ and $A(q,B_2)$, and $\delta^m$
is a trivial upper bound for the measure of one of these regions of
intersection.

If $\Delta qr \geq \gcd(q,r)$ then $B$ can contain at most
\[\left(\frac{\Delta qr}{\gcd(q,r)} + 1\right)^m\] 
integer points of the form $r\mathbf a - q\mathbf b$.
Furthermore, an integer point can
be realized in the form $r\mathbf a-q \mathbf b$ in at most
$\gcd(q,r)^m$ different ways, for if
\begin{equation*}
  r\mathbf a-q \mathbf b = r\mathbf a'-q \mathbf b', 
\end{equation*}
then this implies
\begin{equation*}
\frac{\mathbf a - \mathbf a'}{q}= \frac{\mathbf b- \mathbf b'}{r},
\end{equation*}
and there are $\gcd(q,r)^m$ rational points of this form in
$\I^m$. So, in this case, we can bound
\begin{align*}
  \abs*{A(r,B_1)\cap A(q,B_2)} &\leq \left(\frac{2\Delta q r}{\gcd(q,r)}\right)^m\gcd(q,r)^m\delta^m \\
                               &= 2^m \Delta^m\delta^m q^mr^m \\
                               &= 8^m \psi_1^m\psi_2^m \\
                               &= 2^m \abs{B_1}\abs{B_2}.
\end{align*}

On the other hand, if $\Delta qr < \gcd(q,r)$, then $B$ contains at
most one integer point of the form $r \mathbf a - q\mathbf b$ which (if it exists) can be realized in that form
in at most $\gcd(q,r)^m$ different ways. Therefore, in this case,
$N \leq \gcd(q,r)^m$ and we have
\begin{equation*}
  \abs*{A(r,B_1)\cap A(q,B_2)} \leq \delta^m \gcd(q,r)^m \leq 2^m\frac{\psi_{2}^m}{q^m}\gcd(q,r)^m = \abs{B_2}q^{-m}\gcd(q,r)^m. 
\end{equation*}

Combining the two possible cases, $\Delta qr < \gcd(q,r)$ and
$\Delta qr \geq \gcd(q,r)$, the proof of the lemma is complete.
\end{proof}

\section{Proof of Theorem~\ref{thm:inheritance}}
\label{sec:proof-theor-refthm}

The following definition is the way of quantifying ``weak
quasi-independence on average'' that we described in
Section~\ref{sec:an-indep-inher}. 

\begin{definition}\label{def:weakind}
Given a sequence of balls $\Psi:=(B_q)_{q \in \N}$ in $\R^m / \Z^m$, we will say that the sets
  $\parens*{A_{n,m}(\q, B_{\abs{\q}})}_{\q\in\ZZ^n}$ are $w$-\emph{weakly
  quasi-independent on average} ($w$-QIA, for short) if for $w\geq 0$ we have
  \begin{equation}\label{eq:weakind}
    \limsup_{Q\to\infty} \parens*{\sum_{\abs{\q}=1}^Q \abs{B_{\abs{\q}}}}^2\parens*{\sum_{1\leq \abs{\q_1}\leq\abs{\q_2}\leq Q} \parens*{\frac{\gcd(\abs{\q_1},\abs{\q_2})}{\abs{\q_1}}}^w\abs*{A\parens*{\q_1,B_{\abs{\q_1}}}\cap A\parens*{\q_2,\parens*{\frac{\abs{\q_1}}{\abs{\q_2}}}^{\frac{w}{m}}B_{\abs{\q_2}}}}}^{-1} > 0.
  \end{equation}
  For tidiness, we use the notation
    \begin{equation*}
    \Pi(\q_1,\q_2, \Psi, w) =\Pi_{n,m}(\q_1,\q_2, \Psi, w):=\abs*{A\parens*{\q_1,B_{\abs{\q_1}}}\cap A\parens*{\q_2,\parens*{\frac{\abs{\q_1}}{\abs{\q_2}}}^{\frac{w}{m}}B_{\abs{\q_2}}}}
  \end{equation*}
  when $\abs{\q_1} \leq \abs{\q_2}$, and for integers $q,r \geq 1$, we write
  \begin{equation*}
    \Gamma(q,r) = \frac{\gcd(q,r)}{\min(q,r)}.
  \end{equation*}
  This way, the expression~(\ref{eq:weakind}) becomes
  \begin{equation}\label{eq:weakindnotation}
    \limsup_{Q\to\infty} \parens*{\sum_{\abs{\q}=1}^Q \abs{B_{\abs{\q}}}}^2\parens*{\sum_{1\leq \abs{\q_1}\leq\abs{\q_2}\leq Q} \Gamma(\abs{\q_1}, \abs{\q_2})^w \Pi(\q_1,\q_2, \Psi, w)}^{-1} > 0.
  \end{equation}
  Note that $w$-QIA implies $w'$-QIA whenever $w \leq w'$ and that
  $0$-QIA coincides with QIA.
\end{definition}

\begin{theorem}[Strong Independence Inheritance]\label{thm:stronginheritance}
  Let $n,m\geq 1$  and let $w\geq 0$. Suppose
  $(B_q)_{q\in\N}$ is a sequence of balls in $\RR^m/\ZZ^m$. If the the
  sets
  \begin{equation*}
  \parens*{A_{n,m}(\q, B_{\abs{\q}})}_{\q\in\ZZ^n}
\end{equation*}
are $w$-QIA, then the sets
\begin{equation*}
\parens*{A_{n+k,m}(\q,
  \abs{\q}^{-k/m}B_{\abs{\q}})}_{\q\in\ZZ^{n+k}}
\end{equation*}
are $\max(w-k, 0)$-QIA for every integer $k\geq 0$.
\end{theorem}

\begin{proof}[Proof of Theorem~\ref{thm:stronginheritance}]
  Denote $\Psi_0:=(B_q)_{q\in\N}$, the given sequence of balls in
  $\RR^m/\ZZ^m$, and for integer $k\geq 0$ denote
  $\Psi_k:=(q^{-k/m}B_q)_{q\in\NN}$.

  It suffices to prove the theorem with $k=1$. Furthermore, since
  $0$-QIA implies $1$-QIA, we may take $w\geq 1$. For
  $\q_1, \q_2\in \ZZ^{n+1}$ with $\abs{\q_i}=q_i$ and $q_1\leq q_2$,
  we are concerned with overlaps of the form
  \begin{equation}\label{eq:whatwecareabout}
    \Pi_{n+1,m}(\q_1,\q_2, \Psi_1, w-1):=\abs*{A_{n+1,m}(\q_1, q_1^{-1/m}B_{q_1})\cap A_{n+1,m}\left(\q_2, \parens*{\frac{q_1}{q_2}}^{\frac{w-1}{m}}q_2^{-1/m}B_{q_2}\right)}.
  \end{equation}
  If $\q_1\nparallel \q_2$ then 
  \begin{align*}
    \Pi_{n+1,m}(\q_1, \q_2, \Psi_1, w-1) &= \abs*{A_{n+1,m}(\q_1, q_1^{-1/m}B_{q_1})} \abs*{A_{n+1,m}\left(\q_2, \parens*{\frac{q_1}{q_2}}^{\frac{w-1}{m}}q_2^{-1/m}B_{q_2}\right)} \nonumber \\
    &\leq \abs*{A_{n+1,m}(\q_1, q_1^{-1/m}B_{q_1})} \abs*{A_{n+1,m}(\q_2, q_2^{-1/m}B_{q_2})}
  \end{align*}
  by Lemma~\ref{lem:nparallel}. In particular, we have
  \begin{equation}\label{eq:parallel}
    \sum_{1\leq q_1 \leq q_2 \leq Q}\sum_{\substack{\q_1\nparallel \q_2\in\ZZ^{n+1} \\ |\q_i|=q_i}} \Gamma(q_1,q_2)^{w-1}\Pi_{n+1,m}(\q_1,\q_2, \Psi_1, w-1) \leq \parens*{\sum_{\substack{\q \in\ZZ^{n+1} \\ \abs{\q}\leq Q}} \abs{A_{n+1,m}(\q, \abs{\q}^{-1/m}B_{\abs{\q}})}}^2.
  \end{equation}
  On the other hand, for $\q_1\parallel\q_2$, we have as a consequence
  of Lemma~\ref{lem:invariance} that
  \begin{multline*}
    \sum_{1\leq q_1 \leq q_2 \leq Q}\sum_{\substack{\q_1\parallel \q_2\in\ZZ^{n+1} \\ |\q_i|=q_i}} \Gamma(q_1,q_2)^{w-1}\Pi_{n+1,m}(\q_1,\q_2, \Psi_1, w-1)\\
    \ll  \sum_{1\leq q_1 \leq q_2 \leq Q}\sum_{\substack{\br_1\parallel \br_2\in\ZZ^n \\ \abs{\br_i}=q_i}} \gcd(q_1,q_2) \Gamma(q_1,q_2)^{w-1}\Pi_{n,m}(\br_1,\br_2, \Psi_1, w-1),
  \end{multline*}
  noting that in $\ZZ^d$ the number of parallel pairs with norms
  $q_1, q_2$ is comparable to $\gcd(q_1,q_2)^{d-1}$. Then by
  Lemma~\ref{lem:dilate} we see that we may follow this with
  \begin{equation}\label{eq:forlater}
    \ll \sum_{1\leq q_1 \leq q_2 \leq Q}\sum_{\substack{\br_1\parallel \br_2\in\ZZ^n \\ \abs{\br_i}=q_i}} \Gamma(q_1,q_2)^w \Pi_{n,m}(\br_1,\br_2,\Psi_0, w).
  \end{equation}
  By assumption, there are infinitely many values of $Q \in \N$ for which this last sum
  is
  \begin{equation*}
\ll \parens*{\sum_{\substack{\br\in \ZZ^n \\\abs{\br}\leq Q}}\abs{A(\br,B_{\abs{\br}})}}^2.
  \end{equation*}
Since
    \begin{equation*}
    \sum_{\substack{\q \in\ZZ^{n+1} \\ \abs{\q}\leq Q}} \abs{A_{n+1,m}(\q, \abs{\q}^{-1/m}B_{\abs{\q}})} \asymp \sum_{q=1}^Q q^{n-1} \abs{B_q} \asymp \sum_{\substack{\br\in\ZZ^n \\ \abs{\br}\leq Q}} \abs{A_{n,m}(\br, B_{\abs{\br}})},
  \end{equation*}
  we have shown that there are infinitely many values of $Q \in \N$ for which
  \begin{align*}
\sum_{1\leq q_1 \leq q_2 \leq Q}\sum_{\substack{\q_1\parallel \q_2\in\ZZ^{n+1} \\ |\q_i|=q_i}} \Gamma(q_1,q_2)^{w-1}\Pi_{n+1,m}(\q_1,\q_2, \Psi_1, w-1) \ll \parens*{\sum_{\substack{\q \in\ZZ^{n+1} \\ \abs{\q}\leq Q}} \abs{A_{n+1,m}(\q, \abs{\q}^{-1/m}B_{\abs{\q}})}}^2.
  \end{align*}
  Combining this with~(\ref{eq:parallel}), we see that there is some
  $C>1$ and infinitely many values of $Q \in \N$ for which
  \begin{equation*}
    \sum_{1\leq \abs{\q_1} \leq \abs{\q_2} \leq Q}\Gamma(q_1,q_2)^{w-1} \Pi_{n+1,m}(\q_1,\q_2,\Psi_1, w-1) \leq C \parens*{\sum_{\abs{\q}\leq Q} \abs*{A_{n+1,m}(\q, \abs{\q}^{-1/m}B_{\abs{\q}})}}^2.
  \end{equation*}
  That is, the sets
  $\parens*{A_{n+1,m}(\q, \abs{\q}^{-1/m}B_{\abs{\q}})}_{\q\in\ZZ^{n+1}}$ are
  $(w-1)$-QIA as the theorem claims. The general statement follows by induction.
\end{proof}

\begin{proposition}[Full measure]\label{prop:fullmeasure}
  Let $n,m\geq 1$ and suppose $(B_q)_{q\in\N}$ is a sequence of balls in
  $\RR^m/\ZZ^m$. If the the sets
  $\parens*{A(\q, B_{\abs{\q}})}_{\q\in\ZZ^n}$ are $0$-QIA and the sum
  $\sum_{\q\in\ZZ^n} \abs{B_{\abs{\q}}}$ diverges, then the set
  \begin{equation*}
    \limsup_{\abs{\q}\to\infty}A_{n,m}(\q, B_{\abs{\q}})
  \end{equation*}
  has full Lebesgue measure.
\end{proposition}

\begin{proof}
  Let $U\subset \I^{nm}$ be an open set. Then for infinitely many $Q \in \N$
  we have
    \begin{align*}
      \sum_{|\q_1|, |\q_2| \leq Q} |A(\q_1, B_{\abs{\q_1}})\cap A(\q_2,B_{\abs{\q_2}})\cap U| &\leq C_1 \parens*{\sum_{\abs{\q}\leq Q} \abs*{A(\q, B_{\abs{\q}})}}^2\\                                                                                                                                 &\overset{\textrm{Lemma~\ref{lem:mixing}}}{\leq} \frac{C_2}{\abs{U}^2}\parens*{\sum_{\abs{\q}\leq Q} \abs*{A(\q, B_{\abs{\q}})\cap U}}^2,
  \end{align*}
  where $C_1>0$ is a universal constant coming from the 0-QIA
  assumption and $C_2>0$ is a constant which may depend on $m$. Now,
  as a consequence of Lemma~\ref{lem:borelcantelli}, we see that
    \begin{equation*}
    \abs*{\limsup_{\abs{\q}\to\infty}A_{n,m}(\q, B_{\abs{\q}})\cap U} \geq \frac{\abs{U}^2}{C_2}. 
  \end{equation*}
  The theorem now follows by Lemma~\ref{lem:lebesguedensity}.
\end{proof}

\begin{proof}[Proof of Theorem~\ref{thm:inheritance}]
  The independence statement follows directly from
  Theorem~\ref{thm:stronginheritance}. The full measure statement
  follows from Proposition~\ref{prop:fullmeasure}.
\end{proof}

\section{Proof of Theorem~\ref{thm:generalmarquee}}\label{sec:proof-theor-refthm:m}

\begin{proposition}[Base case for Theorem~\ref{thm:generalmarquee}]\label{prop:basecase}
Suppose $(B_q)_{q\in\NN}$ is a sequence of balls in $\RR^m/\ZZ^m$
  such that $\sum \abs{B_q}$ diverges.
  \begin{itemize}
  \item If $m\geq 3$, then the sets $(A_{1,m}(q, B_q))_{q\in\NN}$ are
    quasi-independent on average.
  \item If $m\geq 2$, then the weaker estimate
    \begin{equation*}
    \sum_{1\leq r \leq q\leq Q} \frac{\gcd(q,r)}{r}\abs*{A_{1,m}(r,B_r)\cap A_{1,m}(q,(r/q)^{1/m}B_q)} \ll \parens*{\sum_{q=1}^Q \abs{A_{1,m}(q,B_q)}}^2,
  \end{equation*}
  holds.
  \item If $m\geq 1$, then the still weaker estimate
    \begin{equation*}
    \sum_{1\leq r \leq q\leq Q} \parens*{\frac{\gcd(q,r)}{r}}^2\abs*{A_{1,m}(r,B_r)\cap A_{1,m}(q,(r/q)^{2/m}B_q)} \ll \parens*{\sum_{q=1}^Q \abs{A_{1,m}(q,B_q)}}^2,
  \end{equation*}
  holds.
\end{itemize}
\end{proposition}

\begin{remark*}
  In the language introduced in Definition~\ref{def:weakind}, this
  proposition says that the sets are
  $\max\set{(3-m),0}$-QIA.
\end{remark*}

\begin{proof}[Proof of Proposition~\ref{prop:basecase}]
  First, suppose $m\geq 3$. By Lemma~\ref{lem:overlaps}, 
  \begin{equation*}
    \sum_{q,r = 1}^Q \abs*{A(r,B_r) \cap A(q,B_q)} \ll \sum_{q,r=1}^Q \left(\abs{B_r}\abs{B_q} + \abs{B_q}q^{-m}\gcd(q,r)^m \right). 
  \end{equation*}
  The second sum on the right-hand side is
  \begin{align*}
    \sum_{q,r = 1}^Q \abs{B_q}q^{-m}\gcd(q,r)^m  &\ll \sum_{q=1}^Q\sum_{r=1}^q\abs{B_q}q^{-m}\gcd(q,r)^m \\
    &= \sum_{q=1}^Qq^{-m}\abs{B_q} \underbrace{\sum_{r=1}^q\gcd(q,r)^m}.
  \end{align*}
  The indicated $\gcd$ sum is $\ll q^m$ if $m\geq 3$. After all,
  \begin{align*}
    \sum_{r=1}^q\gcd(q,r)^m &\leq \sum_{d\mid q}d^m(q/d) \\
                            &=\sum_{d\mid q}(q/d)^m d \\
    &\leq q^m \underbrace{\sum_{i=1}^qi^{-m+1}},
  \end{align*}
  and the indicated sum is absolutely bounded for $m\geq
  3$. Therefore, we have
  \begin{align*}
    \sum_{q,r = 1}^Q \abs*{A(r,B_r) \cap A(q,B_q)} &\ll \sum_{q,r=1}^Q \abs{B_r}\abs{B_q} + \sum_{q=1}^Q \abs{B_q} \\
    &\ll \parens*{\sum_{q=1}^Q \abs{B_q}}^2,
  \end{align*}
  since the series in parentheses diverges.

  Now let us suppose only that $m\geq 2$. Then a nearly identical
  calculation can be done. By Lemma \ref{lem:overlaps}, we have
  \begin{equation*}
    \sum_{1\leq r\leq q\leq Q} \frac{\gcd(q,r)}{r}\abs*{A(r,B_r)\cap A(q,(r/q)^{1/m}B_q)} \ll \sum_{1\leq r\leq q\leq Q} \left(\abs{B_r}\abs{B_q} + \frac{r}{q}\abs{B_q}q^{-m}r^{-1}\gcd(q,r)^{m+1}\right). 
  \end{equation*}
  This time the second sum on the right-hand side is
  \begin{equation*}
    \sum_{1\leq r\leq q\leq Q} \abs{B_q}q^{-m-1}\gcd(q,r)^{m+1}\ll \sum_{q=1}^Qq^{-m-1}\abs{B_q} \underbrace{\sum_{r=1}^q\gcd(q,r)^{m+1}}
  \end{equation*}
  and the $\gcd$ sum is $\ll q^{m+1}$ since we have assumed $m\geq
  2$. Therefore, we have
  \begin{align*}
    \sum_{1\leq r\leq q\leq Q} \frac{\gcd(q,r)}{r}\abs*{A(r,B_r)\cap A(q,(r/q)^{1/m}B_q)} &\ll \sum_{q,r=1}^Q \abs{B_q}\abs{B_r} + \sum_{q=1}^Q \abs{B_q} \\
    &\ll \parens*{\sum_{q=1}^Q \abs{B_q}}^2,
  \end{align*}
  as in the previous paragraph.

  Finally, suppose only that $m\geq 1$. Again by Lemma
  \ref{lem:overlaps}, we have
  \begin{equation*}
    \sum_{1\leq r\leq q\leq Q} \parens*{\frac{\gcd(q,r)}{r}}^2\abs*{A(r,B_r)\cap A(q,(r/q)^{2/m}B_q)} \ll \sum_{1\leq r\leq q\leq Q} \left(\abs{B_r}\abs{B_q} + \parens*{\frac{r}{q}}^2\abs{B_q}q^{-m}r^{-2}\gcd(q,r)^{m+2}\right). 
  \end{equation*}
  This time the second sum on the right-hand side is 
  \begin{equation*}
    \sum_{1\leq r\leq q\leq Q} \abs{B_q}q^{-m-2}\gcd(q,r)^{m+2}\ll \sum_{q=1}^Qq^{-m-2}\abs{B_q} \underbrace{\sum_{r=1}^q\gcd(q,r)^{m+2}}
  \end{equation*}
  and the $\gcd$ sum is $\ll q^{m+2}$ since we have assumed $m\geq 1$. Therefore, we have 
  \begin{align*}
    \sum_{1\leq r\leq q\leq Q} \parens*{\frac{\gcd(q,r)}{r}}^2\abs*{A(r,B_r)\cap A(q,(r/q)^{2/m}B_q)} &\ll \sum_{q,r=1}^Q \abs{B_q}\abs{B_r} + \sum_{q=1}^Q \abs{B_q} \\
    &\ll \parens*{\sum_{q=1}^Q \abs{B_q}}^2,
  \end{align*}
  as before. This finishes the proof.
\end{proof}

\begin{proof}[Proof of Theorem~\ref{thm:generalmarquee}] 
  Let $nm>2$. The convergence part is easily disposed by an
  application of the First Borel--Cantelli lemma, so let us focus on
  the divergence part.

  First, notice that we lose no generality by choosing some $c>0$ and
  assuming that
  \begin{equation}\label{eq:wlog}
    (\forall q\in \NN)\qquad q^{n-1}\abs{B_q} \leq c.
  \end{equation}
  After all, for any $q$ where it does not hold, we can shrink the
  ball until $q^{n-1}\abs{B_q}= c$ and work with the possibly-shrunk
  sequence of balls instead. The divergence condition will still hold
  for the smaller balls. Let us therefore assume that~(\ref{eq:wlog})
  holds with a fixed $c < 1$.

  Now we may obtain a new sequence of balls
  $\widehat{B}_q = q^{(n-1)/m}B_q \subset
  \RR^m/\ZZ^m$. Notice that the sum
  $\sum \abs{\widehat{B}_q} = \sum q^{n-1}\abs{B_q}$
  diverges by assumption, so Proposition~\ref{prop:basecase}
  guarantees that the sets
  \begin{equation*}
    A_{1,m}(q, \widehat{B}_q) = A_{1,m} (q, q^{(n-1)/m} B_q)
  \end{equation*}
  are $\max\set{(3-m),0}$-QIA, in the sense of
  Definition~\ref{def:weakind}. Therefore, by
  Theorem~\ref{thm:stronginheritance} the sets
  \begin{equation*}
A_{n,m}(\q, \abs{\q}^{-(n-1)/m}\widehat{B}_{\abs{\q}})=    A_{n,m}(\q, B_{\abs{\q}}) 
  \end{equation*}
  inherit $\max\set{(3-m)-(n-1), 0}$-QIA, which, since
  $4 - (m+n) \leq 0$, is $0$-QIA. Now, by
  Proposition~\ref{prop:fullmeasure},
  \begin{equation*}
\limsup_{\abs{\q}\to\infty}A_{n,m}(\q, B_{\abs{\q}}) 
  \end{equation*}
  has full measure, and this proves the theorem.
\end{proof}

\section{Proof of Theorem~\ref{thm:remaining}} \label{sec:remaining}

The following theorem is a more general version of
Theorem~\ref{thm:remaining} that does not require the balls to be
concentric.

\begin{theorem}\label{thm:generalextradiv}
  Suppose $nm = 2$ and $\eps>0$. For any sequence of balls
  $\Psi :=(B_q)_{q \in \N}\subset\RR^m/\ZZ^m$, we have
  \begin{equation*}
    \abs*{\calA_{n,m}(\Psi)} = 1 \qquad\textrm{if}\qquad \sum_{q=1}^\infty q^{n-1}\parens*{\frac{\varphi(q)}{q}}^{1+\eps}\abs{B_q} =\infty.
  \end{equation*}
\end{theorem}

\begin{remark*}
  The above theorem only has two cases: $(1,2)$ and $(2,1)$. The
  corresponding divergence conditions are
  \begin{equation*}
    \sum_{q=1}^\infty \parens*{\frac{\varphi(q)}{q}}^{1+\eps}\abs{B_q} =\infty\qquad\textrm{and}\qquad\sum_{q=1}^\infty \parens*{\frac{\varphi(q)}{q}}^{1+\eps} q \abs{B_q} =\infty.
  \end{equation*}
\end{remark*}

\begin{remark*}
Following the arguments outlined in Section \ref{sec:Hausdorff}, one could obtain a Hausdorff measure analogue of Theorem \ref{thm:generalextradiv} akin to statements seen in Section \ref{sec:Hausdorff}. We leave the details to the reader.
\end{remark*}

\begin{proof}[Proof of Theorem~\ref{thm:generalextradiv} for
  $(n,m) = (2,1)$] Suppose $(B_q)_{q \in \N}$ is a sequence of balls
  such that for some $\eps>0$ the series
  \[\sum_{q=1}^{\infty}{\left(\frac{\varphi(q)}{q}\right)^{1+\varepsilon}q|B_q|} \asymp\sum_{\abs{\q}=1}^\infty \parens*{\frac{\varphi(\abs{\q})}{\abs{\q}}}^{1+\eps} \abs{A(\q, B_{\abs{\q}})}=\infty.\]
  By Lemma~\ref{lem:invariance}, we have
    \begin{align}
    \sum_{\substack{1\leq \abs{\q_1} \leq \abs{\q_2}\leq Q \\ \q_1\parallel \q_2}} \abs*{A(\q_1, B_{\abs{\q_1}})\cap A(\q_2, B_{\abs{\q_2}})} &= \sum_{\substack{1\leq \abs{\q_1} \leq \abs{\q_2}\leq Q \\ \q_1\parallel \q_2}} \abs*{A_{1,1}(\abs*{\q_1}, B_{\abs{\q_1}})\cap A_{1,1}(\pm \abs*{\q_2}, B_{\abs{\q_2}})}  \nonumber \\
    &\ll \sum_{1\leq q_1\leq q_2\leq Q} \gcd(q_1,q_2)\abs*{A(q_1, B_{q_1})\cap A(q_2, B_{q_2})}, \label{eq:4}
  \end{align}
  and by Lemma~\ref{lem:dilate}, we  have
    \begin{equation}\label{eq:5}
\abs*{A(q_1, B_{q_1})\cap A(q_2, B_{q_2})} \leq \frac{1}{q_1}\abs*{A(q_1, q_1 B_{q_1})\cap A(q_2, (q_1/q_2) q_2 B_{q_2})}.
\end{equation}
Combining~(\ref{eq:4}) and (\ref{eq:5}) brings us to
  \begin{equation*}
    \sum_{\substack{1\leq \abs{\q_1} \leq \abs{\q_2}\leq Q \\ \q_1\parallel \q_2}} \abs*{A(\q_1, B_{\abs{\q_1}})\cap A(\q_2, B_{\abs{\q_2}})} \ll \sum_{1\leq q_1\leq q_2\leq Q} \frac{\gcd(q_1,q_2)}{q_1}\abs*{A(q_1, q_1 B_{q_1})\cap A(q_2, (q_1/q_2)q_2B_{q_2})},
  \end{equation*}
  and, by Lemma~\ref{lem:overlaps}, we have
  \begin{equation*}
    \abs*{A(q_1, q_1 B_{q_1})\cap A(q_2, (q_1/q_2)q_2B_{q_2})} \ll q_1^2\abs{B_{q_1}}\abs{B_{q_2}} + q_1\abs{B_{q_2}}q_2^{-1}\gcd(q_1,q_2).
  \end{equation*}
  Hence,
\begin{align*}
    \sum_{1\leq q_1\leq q_2\leq Q} \frac{\gcd(q_1,q_2)}{q_1}&\abs*{A(q_1, q_1 B_{q_1})\cap A(q_2, (q_1/q_2)q_2B_{q_2})} 
\\
    &\ll \parens*{\sum_{q=1}^{Q}q\abs{B_q}}^2 +  \sum_{1 \leq q_1\leq q_2 \leq Q} \frac{\gcd(q_1,q_2)^2}{q_2}\abs{B_{q_2}} \\
                                                            &= \parens*{\sum_{q=1}^Q q\abs{B_q}}^2 +  \sum_{q_2=1}^Q\abs{B_{q_2}}q_2^{-1} \sum_{q_1=1}^{q_2} \gcd(q_1,q_2)^2.
\end{align*}
  Now, 
  \begin{equation}\label{eq:now}
    \sum_{r=1}^{q} \gcd(r,q)^2 = \sum_{d\mid q}d^2\varphi\left(\frac{q}{d}\right) 
                               = \sum_{d\mid q}\parens*{\frac{q}{d}}^2 \varphi(d) 
                               = q^2 \sum_{d\mid q}\frac{\varphi(d)}{d^2} 
                               \leq q^2 \sum_{d\mid q}\frac{1}{d}
=  q \sum_{d\mid q}d 
    \ll \frac{q^3}{\varphi(q)},
  \end{equation}
  by~\cite[Theorem~329]{HardyWright}. Therefore, we have 
  \begin{align}
    \sum_{\substack{1\leq \abs{\q_1} \leq \abs{\q_2}\leq Q \\ \q_1\parallel \q_2}} \abs*{A(\q_1, B_{\abs{\q_1}})\cap A(\q_2, B_{\abs{\q_2}})} &\ll \parens*{\sum_{q=1}^Q q\abs{B_q}}^2 +  \sum_{q=1}^Q\abs{B_{q}} \frac{q^2}{\varphi(q)}\nonumber \\
                                                                                                                                              &\ll \parens*{\sum_{\abs{\q}=1}^Q \abs{A(\q, B_{\abs{\q}})}}^2 +  \sum_{\abs{\q}=1}^Q\abs{A(\q, B_{\abs{\q}})} \frac{\abs{\q}}{\varphi(\abs{\q})},\label{eq:1}
  \end{align}
  an estimate we can safely extend to the nonparallel pairs
  $\q_1\nparallel \q_2$, since they are genuinely pairwise independent.
  
  Let $U\subset \I^{2}$ be an open set. Then by a combination
  of~(\ref{eq:1}) and Lemma~\ref{lem:mixing} we have
  \begin{align}
    \sum_{1\leq \abs{\q_1} \leq \abs{\q_2}\leq Q} &\abs*{A(\q_1, B_{\abs{\q_1}})\cap A(\q_2, B_{\abs{\q_2}})\cap U}                           \ll \parens*{\sum_{\abs{\q}=1}^Q \abs{A(\q, B_{\abs{\q}})}}^2 +  \sum_{\abs{\q}=1}^Q\abs{A(\q, B_{\abs{\q}})} \frac{\abs{\q}}{\varphi(\abs{\q})}\nonumber\\
                                                  &\ll \frac{1}{\abs{U}^2}\parens*{\sum_{\abs{\q}=1}^Q \abs{A(\q, B_{\abs{\q}})\cap U}}^2 +  \frac{1}{\abs{U}}\sum_{\abs{\q}=1}^Q\abs{A(\q, B_{\abs{\q}})\cap U} \frac{\abs{\q}}{\varphi(\abs{\q})}\nonumber\\
                                                  &\ll \frac{1}{\abs{U}^2}\brackets*{\parens*{\sum_{\abs{\q}=1}^Q \abs{A(\q, B_{\abs{\q}})\cap U}}^2 +  \sum_{\abs{\q}=1}^Q\abs{A(\q, B_{\abs{\q}})\cap U} \frac{\abs{\q}}{\varphi(\abs{\q})}}.\label{eq:2}
  \end{align}
  The implicit constant does not depend on $U$.

  Following a strategy from the proof of~\cite[Theorem 1.8]{Yu}, we let
  \begin{equation*}
    D_\ell = \set*{q\in \NN : 2^\ell \leq \frac{q}{\varphi(q)} < 2^{\ell+1}}.
  \end{equation*}
  If there exists $\ell\geq 0$ for which
  \begin{equation*}
    \sum_{\abs{\q}\in D_\ell} \parens*{\frac{\varphi(\abs{\q})}{\abs{\q}}}^{1+\eps} \abs{A(\q, B_{\abs{\q}})} = \infty,
  \end{equation*}
  then we can restrict our attention to $D_\ell$. In this case,
  estimate~(\ref{eq:1}) would immediately lead to 0-QIA, and we would then be done by
  Proposition~\ref{prop:fullmeasure}. So assume that there is no such
  $\ell$, and put
  \begin{equation*}
    \Sigma_\ell := \sum_{\abs{\q}\in D_\ell} \parens*{\frac{\varphi(\abs{\q})}{\abs{\q}}}^{1+\eps} \abs{A(\q, B_{\abs{\q}})\cap U} \quad\textrm{and}\quad \Sigma_{\ell,Q} := \sum_{\substack{\abs{\q}\in D_\ell \\ 1 \leq \abs{\q}\leq Q}} \parens*{\frac{\varphi(\abs{\q})}{\abs{\q}}}^{1+\eps} \abs{A(\q, B_{\abs{\q}})\cap U}.
  \end{equation*}
  Notice that $\Sigma_\ell < \infty$ for every $\ell\geq 0$, by
  assumption, and that the $\Sigma_\ell$'s form a divergent series, by
  Lemma~\ref{lem:mixing}. Then, by Lemma~\ref{lem:erdoschung}, we have
  \begin{align*}
    \abs*{\bigcup_{\substack{\abs{\q}\in D_\ell \\ \abs{\q}\leq Q}} A(\q, B_{\abs{\q}})\cap U} &\geq \frac{\parens*{\sum_{\substack{\abs{\q}\in D_\ell \\ \abs{\q}\leq Q}} \abs*{A(\q, B_{\abs{\q}})\cap U}}^2}{\sum_{\substack{1 \leq \abs{\br}, \abs{\q}\leq Q \\ \abs{\br}, \abs{\q}\in D_\ell}}\abs*{A(\q, B_{\abs{\q}})\cap A(\br, B_{\abs{\br}})\cap U}}\\
                                                                                               &\gg \frac{\parens*{\sum_{\substack{\abs{\q}\in D_\ell \\ \abs{\q}\leq Q}} \abs*{A(\q, B_{\abs{\q}})\cap U}}^2}{\sum_{\substack{1\leq \abs{\br}\leq \abs{\q}\leq Q \\ \abs{\br}, \abs{\q}\in D_\ell}} \abs*{A(\q, B_{\abs{\q}})\cap A(\br, B_{\abs{\br}})\cap U}}
  \end{align*}
  for every $\ell$ and $Q$ for which the union on the left-hand side
  has positive measure. This is guaranteed to be the case for
  infinitely many $\ell$, since the measure sum
  diverges. Then, for $Q\geq Q_\ell$ where $Q_\ell$ is
  sufficiently large that the estimates in~(\ref{eq:2}) take effect,
  we have
  \begin{align}
    \abs*{\bigcup_{\substack{\abs{\q}\in D_\ell \\ \abs{\q}\leq Q}} A(\q, B_{\abs{\q}})\cap U} &\gg\abs{U}^2 \parens*{\frac{\parens*{\sum_{\substack{\abs{\q}\in D_\ell \\ \abs{\q}\leq Q}} \abs*{A(\q, B_{\abs{\q}})\cap U}}^2}{\parens*{\sum_{{\substack{\abs{\q}\in D_\ell \\ 1 \leq \abs{\q}\leq Q}}} \abs{A(\q, B_{\abs{\q}})\cap U}}^2 +  \sum_{{\substack{\abs{\q}\in D_\ell \\ 1 \leq \abs{\q}\leq Q}}}\abs{A(\q, B_{\abs{\q}})\cap U} \frac{\abs{\q}}{\varphi(\abs{\q})}}}\nonumber\\
                                                                                                               &= \abs{U}^2 \parens*{1 +  \displaystyle{\frac{\sum_{{\substack{\abs{\q}\in D_\ell \\ 1 \leq \abs{\q}\leq Q}}}\abs{A(\q, B_{\abs{\q}})\cap U} \frac{\abs{\q}}{\varphi(\abs{\q})}}{\parens*{\sum_{{\substack{\abs{\q}\in D_\ell \\ 1 \leq \abs{\q}\leq Q}}} \abs{A(\q, B_{\abs{\q}})\cap U}}^2}}}^{-1} \label{eq:fraction}
  \end{align}
  Note that
  \begin{align*}
    \displaystyle{\frac{\sum_{{\substack{\abs{\q}\in D_\ell \\ 1 \leq \abs{\q}\leq Q}}}\abs{A(\q, B_{\abs{\q}})\cap U} \frac{\abs{\q}}{\varphi(\abs{\q})}}{\parens*{\sum_{{\substack{\abs{\q}\in D_\ell \\ 1 \leq \abs{\q}\leq Q}}} \abs{A(\q, B_{\abs{\q}})\cap U}}^2}} &= \displaystyle{\frac{\sum_{{\substack{\abs{\q}\in D_\ell \\ 1 \leq \abs{\q}\leq Q}}}\abs{A(\q, B_{\abs{\q}})\cap U} \parens*{\frac{\varphi(\abs{\q})}{\abs{\q}}}^{1 + \eps}\parens*{\frac{\abs{\q}}{\varphi(\abs{\q})}}^{2 + \eps}}{\parens*{\sum_{{\substack{\abs{\q}\in D_\ell \\ 1 \leq \abs{\q}\leq Q}}} \abs{A(\q, B_{\abs{\q}})\cap U}\parens*{\frac{\varphi(\abs{\q})}{\abs{\q}}}^{1 + \eps}\parens*{\frac{\abs{\q}}{\varphi(\abs{\q})}}^{1 + \eps}}^2}}\\
    &\leq \frac{(2^{\ell + 1})^{2 + \eps}}{2^{2\ell(1 + \eps)}\Sigma_{\ell,Q}}.
  \end{align*}
  Putting this into~(\ref{eq:fraction}), we find that
  \begin{align*}
    \abs*{\bigcup_{\substack{\abs{\q}\in D_\ell \\ \abs{\q}\leq Q}} A(\q, B_{\abs{\q}})\cap U} &\gg \abs{U}^2 \parens*{\frac{1}{1 +  \frac{2^{2 +\eps}}{2^{\eps\ell}\Sigma_{\ell,Q}}}}.
  \end{align*}
  Now, the fact that the $\Sigma_\ell$ form a divergent series implies
  that there are $\ell$ and corresponding $Q$ for which
  $2^{\eps\ell}\Sigma_{\ell,Q}$ is arbitrarily large. In particular,
  there are infinitely many $\ell \in \N$ (and corresponding
  $Q \in \N$) for which the above string of inequalities gives
  \begin{equation*}
    \underbrace{\abs*{\bigcup_{\substack{\abs{\q}\in D_\ell \\ \abs{\q}\leq Q}} A(\q, B_{\abs{\q}})\cap U}}_{C_\ell} \geq \frac{\abs{U}^2}{2C},
  \end{equation*}
  where $C$ is the implicit constant in the above estimates. Since the
  sets $C_\ell$ all have measure at least $\abs{U}^2/2C$, their
  associated limsup set must have at least that
  measure. Furthermore, since 
  \begin{equation*}
    \limsup_{\ell\to\infty} C_\ell \subset \limsup_{\abs{\q}\to\infty}A(\q, B_{\abs{\q}})\cap U, 
  \end{equation*}
  this implies that
    \begin{equation*}
    \abs*{\limsup_{\abs{\q}\to\infty}A(\q, B_{\abs{\q}})\cap U} \geq \frac{\abs{U}^2}{2C}. 
  \end{equation*}
  The theorem now follows by Lemma~\ref{lem:lebesguedensity}.
\end{proof}

\begin{proof}[Proof of Theorem~\ref{thm:generalextradiv} for $(n,m) = (1,2)$]  Suppose $(B_q)_{q \in \N}$ is a sequence of balls such that for some
  $\eps>0$ the series
  \[\sum_{q=1}^{\infty}{\left(\frac{\varphi(q)}{q}\right)^{1+\varepsilon}|B_q|} \asymp\sum_{\abs{\q}=1}^\infty \parens*{\frac{\varphi(\abs{\q})}{\abs{\q}}}^{1+\eps} \abs{A(\q, B_{\abs{\q}})}=\infty.\] 
  Lemma~\ref{lem:overlaps} gives
\begin{align*}
  \sum_{1\leq r \leq q\leq Q} \abs*{A_{1,2}(r,B_r)\cap A_{1,2}(q,B_q)} &\ll \sum_{1\leq r \leq q\leq Q} \parens*{\abs{B_r}\abs{B_q} +  \abs{B_q}q^{-2}\gcd(r,q)^2} \\
                                                                       &\ll \parens*{\sum_{1\leq q\leq Q} \abs{B_q}}^2 + \sum_{1\leq q\leq Q} \abs{B_q}q^{-2}\sum_{r=1}^q\gcd(r,q)^2 \\
                                                                       &\overset{~(\ref{eq:now})}{\ll} \parens*{\sum_{1\leq q\leq Q} \abs{B_q}}^2 + \sum_{1\leq q\leq Q} \abs{B_q}\frac{q}{\varphi(q)}\\
                                                                       &\ll \parens*{\sum_{q=1}^Q \abs{A(q, B_q)}}^2 +  \sum_{q=1}^Q\abs{A(q, B_q)} \frac{q}{\varphi(q)}.
\end{align*}
Now the rest of the proof follows the proof for the $(2,1)$ case
verbatim, starting at~(\ref{eq:1}) and replacing every instance of
$\q$ with $q$.
\end{proof}

\paragraph{Acknowledgements.} We thank Victor Beresnevich and Sanju
Velani for several helpful comments and for their continued
mathematical support. We would also like to thank Yeni and Stuart for
their patience and support.

\bibliographystyle{plain}


\end{document}